\documentclass[twoside,leqno,10pt, A4]{amsart}
\usepackage{amsfonts}
\usepackage{amsmath}
\usepackage{amscd}
\usepackage{amssymb}
\usepackage{amsthm}
\usepackage{amsrefs}
\usepackage{latexsym}
\usepackage{mathrsfs}
\usepackage{bbm}
\usepackage{enumerate}
\usepackage{graphicx}
\usepackage{amsfonts}
\usepackage{amsmath}
\usepackage{amscd}
\usepackage{amssymb}
\usepackage{amsthm}
\usepackage{amsrefs}
\usepackage{latexsym}
\usepackage{mathrsfs}
\usepackage{bbm}
\usepackage{amscd}
\usepackage{amssymb}
\usepackage{amsthm}
\usepackage{amsrefs}
\usepackage{latexsym}
\usepackage{mathrsfs}
\usepackage{bbm}
\usepackage{enumerate}
\usepackage{graphicx}
\usepackage{color}
\begin{document}

\newtheorem{theorem}[subsection]{Theorem}
\newtheorem{proposition}[subsection]{Proposition}
\newtheorem{lemma}[subsection]{Lemma}
\newtheorem{corollary}[subsection]{Corollary}
\newtheorem{conjecture}[subsection]{Conjecture}
\newtheorem{prop}[subsection]{Proposition}
\numberwithin{equation}{section}
\newcommand{\mr}{\ensuremath{\mathbb R}}
\newcommand{\mc}{\ensuremath{\mathbb C}}
\newcommand{\dif}{\mathrm{d}}
\newcommand{\intz}{\mathbb{Z}}
\newcommand{\ratq}{\mathbb{Q}}
\newcommand{\natn}{\mathbb{N}}
\newcommand{\comc}{\mathbb{C}}
\newcommand{\rear}{\mathbb{R}}
\newcommand{\prip}{\mathbb{P}}
\newcommand{\uph}{\mathbb{H}}
\newcommand{\fief}{\mathbb{F}}
\newcommand{\majorarc}{\mathfrak{M}}
\newcommand{\minorarc}{\mathfrak{m}}
\newcommand{\sings}{\mathfrak{S}}
\newcommand{\fA}{\ensuremath{\mathfrak A}}
\newcommand{\mn}{\ensuremath{\mathbb N}}
\newcommand{\mq}{\ensuremath{\mathbb Q}}
\newcommand{\half}{\tfrac{1}{2}}
\newcommand{\f}{f\times \chi}
\newcommand{\summ}{\mathop{{\sum}^{\star}}}
\newcommand{\chiq}{\chi \bmod q}
\newcommand{\chidb}{\chi \bmod db}
\newcommand{\chid}{\chi \bmod d}
\newcommand{\sym}{\text{sym}^2}
\newcommand{\hhalf}{\tfrac{1}{2}}
\newcommand{\sumstar}{\sideset{}{^*}\sum}
\newcommand{\sumprime}{\sideset{}{'}\sum}
\newcommand{\sumprimeprime}{\sideset{}{''}\sum}
\newcommand{\sumflat}{\sideset{}{^\flat}\sum}
\newcommand{\shortmod}{\ensuremath{\negthickspace \negthickspace \negthickspace \pmod}}
\newcommand{\V}{V\left(\frac{nm}{q^2}\right)}
\newcommand{\sumi}{\mathop{{\sum}^{\dagger}}}
\newcommand{\mz}{\ensuremath{\mathbb Z}}
\newcommand{\leg}[2]{\left(\frac{#1}{#2}\right)}
\newcommand{\muK}{\mu_{\omega}}
\newcommand{\thalf}{\tfrac12}
\newcommand{\lp}{\left(}
\newcommand{\rp}{\right)}
\newcommand{\Lam}{\Lambda_{[i]}}
\newcommand{\lam}{\lambda}
\def\L{\fracwithdelims}
\def\om{\omega}
\def\pbar{\overline{\psi}}
\def\phis{\varphi^*}
\def\lam{\lambda}
\def\lbar{\overline{\lambda}}
\newcommand\Sum{\Cal S}
\def\Lam{\Lambda}
\newcommand{\sumtt}{\underset{(d,2)=1}{{\sum}^*}}
\newcommand{\sumt}{\underset{(d,2)=1}{\sum \nolimits^{*}} \widetilde w\left( \frac dX \right) }

\newcommand{\hf}{\tfrac{1}{2}}
\newcommand{\af}{\mathfrak{a}}
\newcommand{\Wf}{\mathcal{W}}

\theoremstyle{plain}
\newtheorem{conj}{Conjecture}
\newtheorem{remark}[subsection]{Remark}

\makeatletter
\def\widebreve{\mathpalette\wide@breve}
\def\wide@breve#1#2{\sbox\z@{$#1#2$}%
     \mathop{\vbox{\m@th\ialign{##\crcr
\kern0.08em\brevefill#1{0.8\wd\z@}\crcr\noalign{\nointerlineskip}%
                    $\hss#1#2\hss$\crcr}}}\limits}
\def\brevefill#1#2{$\m@th\sbox\tw@{$#1($}%
  \hss\resizebox{#2}{\wd\tw@}{\rotatebox[origin=c]{90}{\upshape(}}\hss$}
\makeatletter

\title[Twisted second moment of modular $L$-functions to a fixed modulus]{Twisted second moment of modular $L$-functions to a fixed modulus}

\author[P. Gao]{Peng Gao}
\address{School of Mathematical Sciences, Beihang University, Beijing 100191, China}
\email{penggao@buaa.edu.cn}

\author[L. Zhao]{Liangyi Zhao}
\address{School of Mathematics and Statistics, University of New South Wales, Sydney NSW 2052, Australia}
\email{l.zhao@unsw.edu.au}

\begin{abstract}
We study asymptotically the twisted second moment of the family of modular $L$-functions to a fixed modulus. As an application, we establish sharp lower
 bounds for all real $k \geq 0$ and sharp upper bounds for $k$ in the range $0 \leq k \leq 1$ for the $2k$-th moment of these $L$-functions on the critical line.
\end{abstract}

\maketitle



\section{Introduction}
\label{sec 1}

 Let $f$ be a fixed holomorphic Hecke eigenform of even weight $\kappa $ and level $1$. Write the Fourier expansion of
  $f$ at infinity as
\[ f(z) = \sum_{n=1}^{\infty} \lambda_f (n) n^{\frac{\kappa -1}{2}} e(nz), \quad \mbox{where} \quad e(z) = \exp (2 \pi i z).  \]
Throughout the paper, let $q$ be a positive integer such that $q \not \equiv 2 \pmod 4$ (so that primitive characters modulo $q$ exist) and $\chi$ a primitive Dirichlet character modulo $q$.  We write $L(s, f \otimes \chi)$ for the twisted modular $L$-function defined in Section \ref{sec:cusp form}. \newline

We aim to evaluate the twisted second moment of the family of modular $L$-functions to the fixed modulus $q$. In \cite{BM15}, V. Blomer and D. Mili\'{c}evi\'{c} studied the second moments of fixed modular $L$-functions at the central point.  In \cite{BFKMMS}, V. Blomer, E. Fouvry, E. Kowalski, P. Michel, D. Mili\'cevi\'c and W. Sawin investigated the twisted second moment of fixed modular $L$-functions  to a fixed prime modulus at more general points. See \cites{Stefanicki,GKR, CH, BM15, BFKMM, KMS17} for other works on the moments of modular $L$-functions.  The motivation of the present work does not solely emanate from \cites{BFKMMS, BM15}.  We also have an application of the twisted moment to establish sharp bounds for the $2k$-th moment of the corresponding family of modular $L$-functions on the critical line. Here, note that upon applying the upper bounds principle due to M. Radziwi{\l\l} and K. Soundararajan  \cite{Radziwill&Sound} and the lower bounds principle due to W. Heap and K. Soundararajan  \cite{H&Sound}, it is shown in \cite{GHH} that we have 
\begin{align*}
\begin{split}
   \sumstar_{\substack{ \chi \shortmod q }}|L(\tfrac{1}{2}, f \otimes \chi)|^{2k} \gg_k & \phis(q)(\log q)^{k^2}, \quad \text{for all } k \geq 0, \quad \mbox{and} \\
    \sumstar_{\substack{ \chi \shortmod q }}|L(\tfrac{1}{2}, f \otimes \chi)|^{2k} \ll_k & \phis(q)(\log q)^{k^2}, \quad \text{for } 0 \leq k \leq 1,
\end{split}
\end{align*}  
  where $\phis(q)$ denotes the number of primitive characters modulo $q$ and $\sum^{\star}$ the sum over primitive Dirichlet characters
  modulo $q$ throughout the paper. \newline

Furthermore, moments of families of $L$-functions on the critical line also attract much attention. For instance, it is shown by M. Munsch \cite{Munsch17} that upper bounds for the shifted moments of the family of Dirichlet $L$-functions to a fixed modulus can be applied to obtain bounds for moments of character sums. Using  a method of K. Soundararajan \cite{Sound2009} and its refinement by A. J. Harper \cite{Harper} on sharp upper bounds for shifted moments of $L$-functions under the generalized Riemann hypothesis (GRH), improvements of Munsch's results were obtained by B. Szab\'o in \cite{Szab}.  In \cite{G&Zhao24-12}, the authors applied a similar approach to show that under GRH, for a large fixed modulus $q$, any positive integer $k$, real tuples ${\bf a} =  (a_1, \ldots, a_k), {\bf t} =  (t_1, \ldots, t_k)$  such that $a_j \geq 0$ and $|t_j| \leq q^A$ for a fixed positive real number $A$,
\begin{align}
\label{Lprodbounds}
\begin{split}
 \sumstar_{\substack{ \chi \bmod q }} & \big| L\big( \tfrac12+it_1, f \otimes \chi \big) \big|^{a_1} \cdots \big| L\big( \tfrac12+it_{k},f \otimes \chi  \big) \big|^{a_{k}} \\
 &  \hspace*{0.5cm} \ll  \varphi(q)(\log q)^{(a_1^2+\cdots +a_{k}^2)/4} \prod_{1\leq j<l \leq k} \Big|\zeta \Big(1+i(t_j-t_l)+\tfrac 1{\log q} \Big) \cdot L\Big(1+i(t_j-t_l)+\tfrac 1{\log q}, \operatorname{sym}^2 f\Big) \Big|^{a_ja_l/2}.
\end{split}
\end{align}
  where $\varphi$ denotes the Euler totient function,  $\zeta(s)$ is the Riemann zeta function and $L(s, \operatorname{sym}^2 f)$  the symmetric square $L$-function of $f$ defined in Section \ref{sec:cusp form}. \newline

Setting $t_j=t$ in \eqref{Lprodbounds} and applying the bound $\zeta(1+\tfrac 1{\log q}) \ll \log q$ (see \cite[Corollary 1.17]{MVa1}), $L(1+\tfrac 1{\log q}, \operatorname{sym}^2 f) \ll 1$ (see Section \ref{sec:cusp form}), we deduce that under GRH,  for any real $k \geq 0$ and $|t| \leq q^A$ for a fixed positive real number $A$,
\begin{align*}
\begin{split}
\sumstar_{\substack{ \chi \shortmod q }}\big| L\big( \tfrac12+it,\chi \big) \big|^{2k} \ll_{{\bf t}, k} & \  \varphi(q)(\log q)^{k^2}. 
\end{split}
\end{align*}  
We shall establish the above as well as its complementary result on lower bounds for certain ranges of ${\bf t}, k$ and certain $q$ unconditionally.   Our results rely crucially on the following evaluation of the twisted second moment of the family of modular $L$-functions to a fixed modulus $q$.  We shall reserve the letter $p$ for a prime number throughout this paper. 
\begin{theorem}
\label{thmtwistedsecondmoment}
    With the notation as above, let $q \not \equiv 2 \pmod 4$ be a positive integer, $a, b$ be positive integers such that $(a,b)=(ab, q)=1$ and $s_1=\sigma_1+it_1$ and $s_2=\sigma_2+it_2$ with $0<\sigma_1$, $\sigma_2<1$, $t_1, t_2\in \mr$, $s_1+s_2 \neq 1$. Suppose that one of the two conditions are satisfied:

\begin{enumerate}[(i)]
\item \label{case1} There exists a divisor $q_0$ of $q$ such that $q/q_0$ is odd and $q^{\eta} \ll q_0 \ll q^{1/2-\eta}$ for some $0<\eta <1/2$.

\item \label{case2} The integer $q$ is a prime and $\eta=1/144$. 
\end{enumerate}

Suppose moreover corresponding to the above two conditions,
\begin{align}
\label{scondition}
\begin{split}
 (|s_1|+1)(|s_2|+1) \ll &  \left\{
\begin{array}
  [c]{lll}
  q^{19\eta/5-\varepsilon_0} \text{ for some } 0<\varepsilon_0< \frac {19\eta}5 , &  \text{in case} \; \eqref{case1}, \\ \\
 q^{1/4-\varepsilon_0} \text{ for some } 0<\varepsilon_0< \frac {1}4 ,  & \text{in case} \; \eqref{case2}.
 \end{array}
\right. \\
 \max(a, b) \leq & \ q^{1/4}, \; \text{in case} \; \eqref{case2} .
\end{split}
\end{align}
  Then we have, 
\begin{align}
\label{twistedsecondmoment}
\begin{split}
   \sumstar_{\substack{ \chi \bmod q }} & L(s_1, f \otimes \chi)L(s_2, f \otimes \overline \chi) \chi(a)\overline \chi(b) \\
=& \frac{\phis(q)}{b^{s_1}a^{s_2}} \zeta(s_1+s_2)L(s_1+s_2, \operatorname{sym}^2 f)H(s_1+s_2; q, a,b)\\
&+ \Big( \frac{q}{2\pi}\Big)^{2(1-s_1-s_2)}\frac{\phis(q)}{a^{1-s_1}b^{1-s_2}}\frac {\Gamma(\frac{\kappa-1}{2}+1-s_1)\Gamma(\frac{\kappa-1}{2}+1-s_2)}{\Gamma(\frac{\kappa-1}{2}+s_1)\Gamma(\frac{\kappa-1}{2}+s_2)} \\
& \hspace*{1cm} \times \zeta(2-s_1-s_2)L(2-s_1-s_2, \operatorname{sym}^2 f)H(2-s_1-s_2; q, a,b)\\
&+O\Big( \Big(\frac {q^2}{ab} \Big)^{1/4-(\sigma_1+\sigma_2)/2+\varepsilon} \frac{q}{b^{\sigma_1}a^{\sigma_2}}+(|s_1|+1)^{1-2\sigma_1}(|s_2|+1)^{1-2\sigma_2}\frac{q^{1+2(1-\sigma_1-\sigma_2)}}{a^{1-\sigma_1}b^{1-\sigma_2}}\Big(\frac {q^2}{ab}\Big)^{-3/4+(\sigma_1+\sigma_2)/2+\varepsilon} \Big) \\
&+O \Big( q^{\varepsilon}(1+(|s_1|+1)^{1-2\sigma_1}(|s_2|+1)^{1-2\sigma_2}q^{2(1-\sigma_1-\sigma_2)})\big (|s_1|+1)(|s_2|+1)q^{2}\big )^{|\sigma_1-1/2|+|\sigma_2-1/2|}\mathcal{R} \Big), 
\end{split}
\end{align}
  where $H(s; q, a, b)=\displaystyle \prod_pH_p(s; q, a, b)$ with 
\begin{align}
\label{Hp}
\begin{split}
H_p(s; q, a, b) =\left\{
 \begin{array}
  [c]{lll}
  \displaystyle{\Big(1-\frac {\lambda_f^2(p)}{p^{s}} \Big) \Big(1-\frac {\lambda_f(p^2)}{p^s}+\frac {\lambda_f(p^2)}{p^{2s}}-\frac {1}{p^{3s}} \Big)}, & \text{if} \; p |  q, \\ \\
 \displaystyle{ \Big(1-\frac {\lambda_f^2(p)}{p^{s}}\Big) \Big(1-\frac {\lambda_f(p^2)}{p^s}+\frac {\lambda_f(p^2)}{p^{2s}}-\frac {1}{p^{3s}} \Big) \sum_{j \geq 0}\frac {\lambda_f(p^{l+js})\lambda_f(p^{js})}{p^{js}}}, & \text{if } p^l \|  ab, \\ \\
  1 , &\text{otherwise },
 \end{array}
 \right.
 \end{split} 
\end{align}
  and
\begin{align*}
\begin{split}
  \mathcal{R}=&  \left\{
\begin{array}
  [c]{lll}
  q^{1-\eta/100}+q^{1-\varepsilon_0} , &  \text{for case} \; \eqref{case1}, \\ \\
 q^{1-\eta/100},  & \text{for case} \; \eqref{case2}.
 \end{array}
\right.
\end{split}
\end{align*}
\end{theorem}

   We shall establish Theorem \ref{thmtwistedsecondmoment} based on the approach in \cite{BM15}. Note that condition \eqref{case1} in Theorem \ref{thmtwistedsecondmoment} is satisfied when $q=q^n_0$ where $q_0$ is an odd prime and $n \geq 3$. We point that the case of $q$ being a prime in Theorem \ref{thmtwistedsecondmoment} has already been established as a special case of \cite[Theorem 1.18]{BFKMMS}, except the asymptotical formula here is more explicit and $s_1$, $s_2$ more general. Note that the $O$-terms in \eqref{twistedsecondmoment} are not necessarily smaller than the main term in the $q$-aspect. However, if both $s_1$, $s_2$ are on the critical line, \eqref{twistedsecondmoment} does yield a valid asymptotic formula, which we summarize in the following.
\begin{corollary}
\label{cortwistedsecondmoment}. Under the same conditions and notations of Theorem~\ref{thmtwistedsecondmoment}, we have
\begin{align}
\label{twistedsecondmomentsspecial}
\begin{split}
 \sumstar_{\substack{ \chi \bmod q }}& L(\tfrac 12+it_1, f \otimes \chi)L(\tfrac 12+it_2, f \otimes \overline \chi) \chi(a)\overline \chi(b) \\
=& \frac{\phis(q)}{b^{1/2+it_1}a^{1/2+it_2}} \zeta(1+i(t_1+t_2))L(1+i(t_1+t_2), \operatorname{sym}^2 f)H(1+i(t_1+t_2); q, a,b)\\
& \hspace*{0.5cm} + (\frac{q}{2\pi})^{-2i(t_1+t_2)}\frac{\phis(q)}{a^{1/2-it_1}b^{1/2-it_2}}\frac {\Gamma(\frac{\kappa}{2}-it_1)\Gamma(\frac{\kappa}{2}-it_2)}{\Gamma(\frac{\kappa}{2}+it_1)\Gamma(\frac{\kappa}{2}+it_2)} \\
& \hspace*{1cm} \times \zeta(1-i(t_1+t_2))L(1-i(t_1+t_2), \operatorname{sym}^2 f)H(1-i(t_1+t_2); q, a,b)  +O \Big( \Big(\frac {q^2}{ab} \Big)^{-1/4+\varepsilon} \frac{q}{\sqrt{ab}}+ q^{\varepsilon}\mathcal{R} \Big).  
\end{split}
\end{align}
\end{corollary}
 
   We now consider the case $t_1 \neq 0$ and $t_2\rightarrow -t_1$ in \eqref{twistedsecondmomentsspecial}. Note that by \cite[Corollary 1.17]{MVa1}, one has
\begin{align*}
\begin{split}
 \zeta(1+it)=\frac 1{it}+O(1), \quad |t| \leq 1.
\end{split}
\end{align*}        
   
   The above estimations allow us to see that the right-hand side of \eqref{twistedsecondmomentsspecial} remains holomorphic in the process that $t_1 \neq 0$ and $t_2 \rightarrow t_1$.  Moreover,
\begin{align*}
\begin{split}   
   \Big( \frac{q}{2\pi} \big)^{-2i(t_1+t_2)} =& 1-2i(t_1+t_2)\log \Big(\frac{q}{2\pi} \Big) +O((t_1+t_2)^2), \\
   a^{-1/2-it_2}= & a^{-1/2+it_1-i(t_1+t_2)}=a^{-1/2+it_1}(1- i(t_1+t_2)\log a)+O((t_1+t_2)^2), \\ 
   b^{-1/2+it_2}= & b^{-1/2-it_1+i(t_1+t_2)}=b^{-1/2-it_1}(1 +i(t_1+t_2)\log b)+O((t_1+t_2)^2), \quad \mbox{and} \\
   \frac {\Gamma(\frac{\kappa}{2}-it_2)}{\Gamma(\frac{\kappa}{2}+it_2)} =&  \frac {\Gamma(\frac{\kappa}{2}+it_1-i(t_1+t_2))}{\Gamma(\frac{\kappa}{2}-it_1+i(t_1+t_2))} \\
   = & \frac {\Gamma(\frac{\kappa}{2}+it_1)}{\Gamma(\frac{\kappa}{2}-it_1)}\Big (1-i\frac {\Gamma(\frac{\kappa}{2}-it_1)\Gamma'(\frac{\kappa}{2}+it_1)+\Gamma'(\frac{\kappa}{2}-it_1)\Gamma(\frac{\kappa}{2}+it_1)}{\Gamma(\frac{\kappa}{2}+it_1)\Gamma(\frac{\kappa}{2}-it_1)}(t_1+t_2)\Big )   +O((t_1+t_2)^2).
\end{split}
\end{align*}   
  We then derive from the above the following special case $t_2 \rightarrow t_1$ of Corollary \ref{cortwistedsecondmoment}. 
\begin{corollary}
\label{cortwistedsecondmomenttspecial}
Under the same conditions and notations of Theorem~\ref{thmtwistedsecondmoment}, we have
\begin{align}
\label{twistedsecondmomenttspecial}
\begin{split}
 \sumstar_{\substack{ \chi \bmod q }}& \Big |L(\frac 12+it, f \otimes \chi)\Big |^2 \chi(a)\overline \chi(b) \\
=& \frac{\phis(q)}{a^{1/2-it} b^{1/2+it}}L(1, \operatorname{sym}^2 f)H(1; q, a,b) \\
& \hspace*{1cm} \times \Big (2\log \Big(\frac{q}{2\pi}\Big)+2L'(1, \operatorname{sym}^2 f)+2H'(1; q, a,b) + 
\frac {\Gamma'(\frac{\kappa}{2}+it)}{\Gamma(\frac{\kappa}{2}+it)}+\frac {\Gamma'(\frac{\kappa}{2}-it)}{\Gamma(\frac{\kappa}{2}-it)}-\log (ab) \Big )\\
& \hspace*{3cm}+O\Big( \Big(\frac {q^2}{ab} \Big)^{-1/4+\varepsilon} \frac{q}{\sqrt{ab}}+ q^{\varepsilon}\mathcal{R} \Big).  
\end{split}
\end{align}
\end{corollary}

Using Corollary \ref{cortwistedsecondmomenttspecial}, we establish sharp bounds for the $2k$-th moment of the family of modular $L$-functions to a fixed modulus on the critical line.  The lower bounds are as follows.
\begin{theorem}
\label{thmlowerbound}
With the notation as above, suppose that one of the two conditions are satisfied:

\begin{enumerate}[(i)]
\item \label{casei} There exists a divisor $q_0$ of $q$ such that $q/q_0$ is odd and $q^{\eta} \ll q_0 \ll q^{1/2-\eta}$ for some $0<\eta <1/2$.

\item \label{caseii} The integer $q$ is a prime and $\eta=1/144$. 
\end{enumerate}
  
  Let $t \in \mr$ such that corresponding to the above two conditions, 
\begin{align}
\label{scondition1}
\begin{split}
 & |t| \ll \left\{
\begin{array}
  [c]{lll}
  q^{19\eta/10-\varepsilon_0} \text{ for some } 0<\varepsilon_0< \frac {19\eta}5 , & \text{in case} \; \eqref{casei}, \\
 q^{1/8-\varepsilon_0} \text{ for some } 0<\varepsilon_0< \frac {1}4 ,  & \text{in case} \; \eqref{caseii}.
 \end{array}
\right.
\end{split}
\end{align}
 
 Suppose moreover that for a fixed $\varepsilon>0$,
\begin{align}
\label{pcondition}
 \sum_{\substack{p|q \\ p \leq q^{\varepsilon}}}\frac {\lambda_f^2(p)}p \ll 1.
\end{align}
 Then for any real number $k \geq 0$, we have
\begin{align*}
   \sumstar_{\substack{ \chi \bmod q }}|L(\tfrac{1}{2}+it, f \otimes \chi)|^{2k} \gg_k \phis(q)(\log q)^{k^2}.
\end{align*}
\end{theorem}

Now we have the following upper bound.
\begin{theorem}
\label{thmupperbound}
Under the same notations and conditions of Theorem~\ref{thmlowerbound}.  For any real number $k$ such that $0 \leq k \leq 1$, we have
\begin{align*}
   \sumstar_{\substack{ \chi \bmod q }}|L(\tfrac{1}{2}+it,f \otimes \chi)|^{2k} \ll_k \phis(q)(\log q)^{k^2}.
\end{align*}
\end{theorem}

Combining Theorems \ref{thmlowerbound} and \ref{thmupperbound} leads to the following result.
\begin{theorem}
\label{thmorderofmag}
Under the same notations and conditions of Theorem~\ref{thmlowerbound}.  Then for any real number $k$ such that $0 \leq k \leq 1$, we have
\begin{align*}
   \sumstar_{\substack{ \chi \bmod q }}|L(\tfrac{1}{2}+it, f \otimes \chi)|^{2k} \asymp \phis(q)(\log q)^{k^2}.
\end{align*}
\end{theorem}

Note that the condition \eqref{pcondition} holds when $q$ has a fixed number of prime divisors.  Our strategy of proofs for Theorems \ref{thmlowerbound} and \ref{thmupperbound} are an extension of \cite{GHH} in treating the special case of $t=0$, which largely follows from the lower bounds principle of W. Heap and K. Soundararajan \cite{H&Sound} and the upper bounds principle of M. Radziwi{\l\l} and K. Soundararajan \cite{Radziwill&Sound} on moments of general families of $L$-functions. 

\section{Preliminaries}
\label{sec 2}

\subsection{Cusp form $L$-functions}
\label{sec:cusp form}

  For any primitive Dirichlet character $\chi$ modulo  $q$, the twisted modular $L$-function $L(s, f \otimes \chi)$ is defined for $\Re(s) > 1$  to be
\begin{align}
\label{Lphichi}
L(s, f \otimes \chi) &= \sum_{n=1}^{\infty} \frac{\lambda_f(n)\chi(n)}{n^s}
 = \prod_{p\nmid q} \left(1 - \frac{\lambda_f (p) \chi(p)}{p^s}  + \frac{1}{p^{2s}}\right)^{-1}=\prod_{p\nmid q} \left(1 - \frac{\alpha_p \chi(p)}{p^s} \right)^{-1}\left(1 - \frac{\beta_p \chi(p)}{p^s} \right)^{-1}.
\end{align}
From Deligne's proof \cite{D} of the Weil conjecture,
\begin{align*}
|\alpha_{p}|=|\beta_{p}|=1, \quad \alpha_{p}\beta_{p}=1.
\end{align*}
  We deduce from this that $\lambda_f (n) \in \mr$ satisfying $\lambda_f (1) =1$ and 
\begin{align}
\label{lambdafbound}
|\lambda_f(n)|\leq d(n), \quad n \geq 1, 
\end{align}
 where $d(n)$ is the number of positive divisors of $n$. \newline

 Let $\iota_{\chi}=i^{\kappa}\tau(\chi)^2/q$, where $\tau(\chi)$ is the Gauss sum associated to $\chi$. We note that the $L$-function $L(s, f \otimes \chi)$ admits
analytic continuation to an entire function and satisfies the
functional equation (\cite[Proposition 14.20]{iwakow})
\begin{align}
\label{fneqn}
  \Lambda(s, f \otimes \chi) = \iota_{\chi}\Lambda(1-s, f \otimes \overline{\chi}), \quad \mbox{where} \; \Lambda(s, f \otimes \chi) = (\frac {q}{2\pi})^{s}\Gamma \big(\frac{\kappa-1}{2}+s \big)L(s, f\otimes \chi).
\end{align}

  Recall that the Rankin-Selberg $L$-function $L(s, f \times f)$ of $f$ is defined (see \cite[(23.24)]{iwakow}) for $\Re(s)>1$ to be
\begin{align}
\label{Lff}
 L(s, f \times f)=\sum_{n \geq 1}\frac {\lambda_f^2(n)}{n^s}=\zeta(s)L(s, \operatorname{sym}^2 f),
\end{align}
  where the last equality above follows from \cite[(5.97)]{iwakow}. Here $L(s, \operatorname{sym}^2 f)$ is the symmetric square $L$-function of $f$ defined for $\Re(s)>1$ by (see \cite[(25.73)]{iwakow})
\begin{align*}
 L(s, \operatorname{sym}^2 f)=& \zeta(2s) \sum_{n \geq 1}\frac {\lambda_f(n^2)}{n^s}=\prod_{p}\Big( 1-\frac {\lambda_f(p^2)}{p^s}+\frac {\lambda_f(p^2)}{p^{2s}}-\frac {1}{p^{3s}} \Big)^{-1}.
\end{align*}
  A result of G. Shimura \cite{Shimura} implies the corresponding completed $L$-function
\begin{align*}
 \Lambda(s, \operatorname{sym}^2 f)=& \pi^{-3s/2}\Gamma \big(\frac {s+1}{2}\big)\Gamma \big(\frac {s+\kappa-1}{2}\big) \Gamma \big(\frac {s+\kappa}{2}\big) L(s, \operatorname{sym}^2 f)
\end{align*}
  is entire and satisfies the functional equation
\begin{align}
\label{fcneqsymsquareLfcn}
 \Lambda(s, \operatorname{sym}^2 f)=\Lambda(1-s, \operatorname{sym}^2 f).
\end{align}
Thus $L(s, f \times f)$ has a simple pole at $s=1$. \newline

 We apply \eqref{fneqn}, \eqref{fcneqsymsquareLfcn} with \cite[(5.8)]{iwakow} and make use of the convexity bounds (see \cite[Exercise 3, p.
  100]{iwakow}) for $L$-functions to get that, for $0 \leq \Re(s) \leq 1$,
\begin{equation} \label{Lsymest}
  L(s,f) \ll  (|s|+\kappa)^{1-\Re(s)+\varepsilon} \quad \mbox{and} \quad  L(s, \operatorname{sym}^2 f) \ll \left( 1+|s| \right)^{3(1-\Re(s))/2+\varepsilon}.
\end{equation}

Furthermore, the convexity bound for $\zeta(s)$ implies that
\begin{align}
\label{zetaest}
\begin{split}
  \zeta(s) \ll & \left( 1+|s| \right)^{(1-\Re(s)/2+\varepsilon}, \quad \mbox{for} \; 0 \leq \Re(s) \leq 1.
\end{split}
\end{align}

\subsection{The approximate functional equation}
\label{sect: apprfcneqn}
 In this section, we develope the approximate functional equations for $ L(\half+it, f \otimes \chi)$ and $L(s_1, f \otimes \chi)L(s_2, f \otimes \overline \chi)$.
\begin{lemma}
\label{PropDirpoly}
   For $X>0$, we have
\begin{align}
\label{lstapprox}
 L(\half+it, f \otimes \chi) = \sum^{\infty}_{n=1} \frac{\lambda_f(n)\chi(n)}{n^{1/2+it}} W_t\left(\frac {n}{qX}\right)+\iota_{\chi}\frac {(2\pi)^{2it}}{q^{2it}}\frac {\Gamma(\frac{\kappa}{2}-it)}{\Gamma(\frac{\kappa}{2}+it)}\sum^{\infty}_{n=1} \frac{\lambda_f(n)\overline \chi(n)}{n^{1/2-it}} W_{-t}\left(\frac {nX}{q}\right),
\end{align}
  where
$$ W_{\pm t}(x) = \frac{1}{2\pi i} \int\limits_{(2)} \frac{\Gamma (\frac{\kappa}{2}\pm it +s)}{\Gamma (\frac{\kappa}{2}\pm it)} e^{s^2}(2\pi x)^{-s} \> \frac{\dif s}{s}.$$
Moreover,
\begin{align}
\label{lsquareapprox}
\begin{split}
  L &(s_1, f \otimes \chi)L(s_2, f \otimes \overline \chi) \\
& =   \sum^{\infty}_{m, n=1} \frac{\lambda_f(m)\lambda_f(n)\chi(m) \overline{\chi}(n)}{m^{s_1}n^{s_2}} \Wf_{s_1,s_2} \left(\frac {mn}{q^2}\right) \\
  & \hspace*{1cm} +\Big(\frac{q}{2\pi}\Big)^{2(1-s_1-s_2)}\frac {\Gamma(\frac{\kappa-1}{2}+1-s_1)\Gamma(\frac{\kappa-1}{2}+1-s_2)}{\Gamma(\frac{\kappa-1}{2}+s_1)\Gamma(\frac{\kappa-1}{2}+s_2)}\sum^{\infty}_{m, n=1} \frac{\lambda_f(m)\lambda_f(n)\overline\chi(m) \chi(n)}{m^{1-s_1}n^{1-s_2}}\Wf_{1-s_1, 1-s_2}\left(\frac {mn}{q^2}\right),
\end{split}
\end{align}
  where
$$ \Wf_{s_1,s_2} (x) = \frac{1}{2\pi i} \int\limits_{(2)} \frac{\Gamma\left(\frac{\kappa-1}{2} +s_1+ s \right)\Gamma\left(\frac{\kappa-1}{2} +s_2+ s \right)}{(2\pi)^{2s} \Gamma\left(\frac{\kappa-1}{2}+s_1 \right)\Gamma\left(\frac{\kappa-1}{2}+s_2 \right) } e^{s^2} x^{-s} \> \frac{\dif s}{s}.$$
  Moreover, the functions $W_{\pm t}(x), \Wf_{s_1,s_2} (x)$ satisfy the bound that for any $c>0$,
\begin{align}
\label{W}
 W_{\pm t}(x) \ll_c \min( 1 , (|t|+1)^cx^{-c}), \quad  \Wf_{s_1, s_2} (x)  \ll_c \min( 1 , (|s_1|+1)^{c}(|s_2|+1)^cx^{-c}).
\end{align}
\end{lemma}
\begin{proof}
  The approximate functional equation given in \eqref{lstapprox} can be derived using standard arguments as those in the proof of \cite[Theorem 5.3]{iwakow}. It remains to establish the approximate functional equation given in \eqref{lsquareapprox}. \newline
  
  We write ${\bf s}=(s_1, s_2), 1-{\bf s}=(1-s_1, 1-s_2)$.  Let $G(s)$ be an entire, even function, bounded in any strip $-A \leq \Re(s) \leq A$ for some $A>2$ such that $G(0)=1$. For a real parameter $X>0$, consider the integral
\begin{align*}
   I(X, {\bf s}, \chi)=\frac 1{2 \pi i}\int\limits_{(2)}\Lambda(s_1+u, f \otimes \chi)\Lambda(s_2+u, f \otimes \overline{\chi})G(u) X^u\frac {\dif u}{u}.
\end{align*}
    Moving the contour of integral to $\Re(u)=-2$, we get
\begin{equation*}
   \Lambda(s_1, f \otimes \chi)\Lambda(s_2, f \otimes \overline{\chi})=I(X,{\bf s}, \chi)-
\frac 1{2 \pi i}\int\limits_{(-2)}\Lambda(s_1+u, f \otimes \chi)\Lambda(s_2+u, f \otimes \overline{\chi})G(u) X^u \frac {\dif u}{u}.
\end{equation*}
Now the functional equation \eqref{fneqn} leads to
\begin{align}
\label{Lambda}
\begin{split}
     \Lambda(s_1, f \otimes \chi)\Lambda(s_2, f \otimes \overline{\chi}) = & I(X, {\bf s}, \chi)
-\frac {\iota_{\chi}\iota_{\overline{\chi}}}{2 \pi i}\int\limits_{(-2)}\Lambda(1-s_1-u, f \otimes \overline \chi)\Lambda(1-s_2-u, f \otimes \chi)G(u)X^u\frac {\dif u}{u} \\
=& I(X, {\bf s},\chi)+ \frac {\iota_{\chi}\iota_{\overline{\chi}}}{2 \pi i}\int\limits_{(2)}\Lambda(1-s_1+u, f \otimes \overline{\chi})\Lambda(1-s_2+u, f \otimes \chi )G(u)X^{-u} \frac {\dif u}{u} \\
=& I(X, {\bf s}, \chi)+I(X^{-1},1-{\bf s}, \overline{\chi}),
\end{split}
\end{align}
  where the penultimate equality above emerges from a change of variable $u \rightarrow -u$ in the first integral, and the last equality follows from \cite[Theorems 9.5, 9.7]{MVa1} and the assumption that $\kappa$ is even so that 
\begin{align*}
\begin{split}
 \iota_{\chi}\iota_{\overline{\chi}}=(-1)^{\kappa}\tau(\chi)^2\tau(\overline \chi)^2/q^2=\tau(\chi)^2(\overline{\tau( \chi)}\chi(-1))^2/q^2=1.
\end{split}
\end{align*}

   Upon expanding $\Lambda(s_i+u, f \otimes \chi), 1 \leq i \leq 2$ into convergent Dirichlet series, we have
\begin{align*}
 I(X, {\bf s},  \chi)=&\frac 1{2 \pi i}\int\limits_{(2)} \left ( \sum_{m_1,  m_2 } \frac{\lambda_f(m_1)\chi(m_1)}{m^{s_1+u}_1}\frac{\lambda_f(m_2)\overline \chi(m_2)}{m^{s_2+u}_2}\prod^2_{j=1} (\frac{q}{2\pi})^{s_j+u-1/2}\Gamma(\frac{\kappa-1}{2}+s_j+u)
 \right )
 G(u) X^u\frac {\dif u}{u}, \\
I(X^{-1}, 1-{\bf s}, \overline{\chi})=&
\frac 1{2 \pi i}\int\limits_{(2)}\left ( \sum_{m_1,  m_2 } \frac{\lambda_f(m_1)\overline  \chi(m_1)}{m^{1-s_1+u}_1}\frac{\lambda_f(m_2) \chi(m_2)}{m^{1-s_2+u}_2}\prod^2_{j=1} (\frac{q}{2\pi})^{1-s_j+u-1/2}\Gamma(\frac{\kappa-1}{2}+1-s_j+u)
 \right )
 G(u)X^{-u} \frac {\dif u}{u}.
\end{align*}

   Applying these expressions and dividing through $\prod^2_{j=1} (\frac{q}{2\pi})^{s_j-1/2}\Gamma(\frac{\kappa-1}{2}+s_j)$ on both sides of \eqref{Lambda},
we obtain that
\begin{align*}
\begin{split}
 L(s_1,& f \otimes\chi)L(s_2,f \otimes \overline \chi)\\
 =& \frac 1{2 \pi i}\int\limits_{(2)} \left ( \sum_{m_1,  m_2 }  \frac{\lambda_f(m_1)\chi(m_1)}{m^{1-s_1+u}_1}\frac{\lambda_f(m_2)\overline \chi(m_2)}{m^{1-s_2+u}_2}\prod^2_{j=1}\Big(\frac{q}{2\pi}\Big)^{u}\frac {\Gamma(\frac{\kappa-1}{2}+s_j+u)}{\Gamma(\frac{\kappa-1}{2}+s_j)}
 \right )
 G(u) X^u\frac {\dif u}{u} \\
 & \hspace*{1cm} + \frac {1}{2 \pi i}\int\limits_{(2)}\left (  \sum_{m_1,  m_2 } \frac{\lambda_f(m_1)\overline \chi(m_1)}{m^{1-s_1+u}_1}\frac{\lambda_f(m_2) \chi(m_2)}{m^{1-s_2+u}_2} \prod^2_{j=1} \Big(\frac{q}{2\pi}\Big)^{1-2s_j+u}\frac {\Gamma(\frac{\kappa-1}{2}+1-s_j+u)}{\Gamma(\frac{\kappa-1}{2}+s_j)}
 \right )
 G(u)X^{-u} \frac {\dif u}{u}.
\end{split}
\end{align*}

   Upon setting $G(u)=e^{u^2}$ and $X=1$, we deduce readily \eqref{lsquareapprox} from the above.  Next we note that Stirling's formula as given in \cite[(5.112)]{iwakow} implies that for $j=1,2$,
\begin{align}
\label{Stirlingratiosquare}
  \frac{\Gamma\left(\frac{\kappa-1}{2} +s_j+ s \right)}{ \Gamma\left(\frac{\kappa-1}{2}+s_j \right) } \ll \frac {|s+s_j|^{\sigma_j+\Re(s)-1/2}}{|s_j|^{\sigma_j-1/2}}\exp \Big(\frac {\pi}{2}(|s_j|-|s_j+s|)\Big) \ll (|s_j|+4)^{\Re(s)}\exp\Big(\frac {\pi}{2}|s|\Big).
\end{align}
  We apply this and argue in a manner similar to the proof of \cite[Propsition 5.4]{iwakow} to see that the bounds in \eqref{W} holds.  This completes the proof of the lemma.
\end{proof}

\section{Proofs of Theorem \ref{thmtwistedsecondmoment}}
\label{sec 3}

With $\mu$ denoting the M\"obius function, we have (see \cite[(3.8)]{iwakow}), for $(a, q)=1$,
\begin{align}
\label{sumchistar}
\sumstar_{\substack{ \chi \bmod q }}\chi(a)=\sum_{c | (q, a-1)}\mu\Big( \frac{q}{c} \Big)\varphi(c).
\end{align}
   In particular, taking $a=1$ in \eqref{sumchistar} gives
\begin{align}
\label{chistar}
 \phis(q)=\sum_{c | q}\mu\Big( \frac{q}{c} \Big) \varphi(c).
\end{align}
  
 We apply \eqref{lsquareapprox} and \eqref{sumchistar}  to see that
\begin{align}
\label{LNsquaresum1}
\begin{split}
 \sumstar_{\substack{ \chi \bmod q }} & L(s_1, f \otimes \chi)L(s_2, f \otimes \overline \chi) \chi(a)\overline \chi(b) \\
=&  \sum^{\infty}_{m, n=1} \frac{\lambda_f(m)\lambda_f(n)}{m^{s_1}n^{s_2}} \Wf_{s_1,s_2} \left(\frac {mn}{q^2}\right)
{\sumstar_{\chi \bmod q}} \chi(ma) \overline{\chi}(nb) \\
& \hspace*{0.5cm} +\Big(\frac{q}{2\pi}\Big)^{2(1-s_1-s_2)}\frac {\Gamma(\frac{\kappa-1}{2}+1-s_1)\Gamma(\frac{\kappa-1}{2}+1-s_2)}{\Gamma(\frac{\kappa-1}{2}+s_1)\Gamma(\frac{\kappa-1}{2}+s_2)}\sum^{\infty}_{m, n=1} \frac{\lambda_f(m)\lambda_f(n)}{m^{1-s_1}n^{1-s_2}}\Wf_{1-s_1, 1-s_2}\left(\frac {mn}{q^2}\right)
{\sumstar_{\chi \bmod q}} \chi(na) \overline{\chi}(mb)\\
=& \sum_{c | q}\mu \Big( \frac{q}{c} \Big)\phi(c) \sum_{\substack{m,n \\ (mn, q)=1 \\ ma \equiv nb \bmod c}}
\frac{\lambda_f(m)\lambda_f(n)}{m^{s_1}n^{s_2}} \Wf_{s_1,s_2} \left(\frac {mn}{q^2}\right) \\
& \hspace*{0.5cm} +\Big(\frac{q}{2\pi}\Big)^{2(1-s_1-s_2)}\frac {\Gamma(\frac{\kappa-1}{2}+1-s_1)\Gamma(\frac{\kappa-1}{2}+1-s_2)}{\Gamma(\frac{\kappa-1}{2}+s_1)\Gamma(\frac{\kappa-1}{2}+s_2)}\sum_{c | q}\mu\Big( \frac{q}{c} \Big)\phi(c) \sum_{\substack{m,n \\ (mn, q)=1 \\ mb \equiv na \bmod c}} \frac{\lambda_f(m)\lambda_f(n)}{m^{1-s_1}n^{1-s_2}} \Wf_{1-s_1,1-s_2} \left(\frac {mn}{q^2}\right).
\end{split}
\end{align}

Now Stirling's formula as given in \cite[(5.112)]{iwakow} implies that 
 \begin{align}
\label{Stirlingratiosquare1}
  \frac{\Gamma\left(\frac{\kappa-1}{2}+1-s_1 \right)\Gamma\left(\frac{\kappa-1}{2}+1-s_2 \right)}{ \Gamma\left(\frac{\kappa-1}{2}+s_1 \right)\Gamma\left(\frac{\kappa-1}{2}+s_2 \right) } \ll (|s_1|+1)^{1-2\sigma_1}(|s_2|+1)^{1-2\sigma_2}.
\end{align}

   It follows from this that the contribution of the case $c=1$ in the in the last display of \eqref{LNsquaresum1} is
\begin{align}
\label{lambdasumc=1'}
\begin{split}
 \ll  \Bigg| \sum_{\substack{m,n \\ (mn, q)=1 }} &
 \frac{\lambda_f(m)\lambda_f(n)}{m^{s_1}n^{s_2}} \Wf_{s_1,s_2} \left(\frac {mn}{q^2}\right) \Bigg| \\ & +(|s_1|+1)^{1-2\sigma_1}(|s_2|+1)^{1-2\sigma_2}q^{2(1-\sigma_1-\sigma_2)}\Bigg|\sum_{\substack{m,n \\ (mn, q)=1 }} \frac{\lambda_f(m)\lambda_f(n)}{m^{1-s_1}n^{1-s_2}} \Wf_{1-s_1,1-s_2} \left(\frac {mn}{q^2}\right)\Bigg|.
\end{split}
\end{align}
We apply \eqref{Lphichi} and  the definition of $\Wf_{s_1, s_2} (x)$ given in Lemma \ref{PropDirpoly} to arrive at 
\begin{align}
\label{lambdasumc=1}
\begin{split}
 \sum_{\substack{m,n \\ (mn, q)=1 }} & \frac{\lambda_f(m)\lambda_f(n)}{m^{s_1}n^{s_2}} \Wf_{s_1,s_2} \left(\frac {mn}{q^2}\right) \\
 =& \frac 1{2\pi i}\int\limits_{(2)}
\frac{\Gamma\left(\frac{\kappa-1}{2} +s_1+ s \right)\Gamma\left(\frac{\kappa-1}{2} +s_2+ s \right)}{(2\pi)^{2s} \Gamma\left(\frac{\kappa-1}{2}+s_1 \right)\Gamma\left(\frac{\kappa-1}{2}+s_2 \right) } L(s_1+s, f)L(s_2+s,f) \\
& \hspace*{2cm} \times \left(1 - \frac{\lambda_f (q)}{q^{s_1+s}} + \frac{1}{q^{2s_1+2s}}\right )\left(1 - \frac{\lambda_f (q)}{q^{s_2+s}} + \frac{1}{q^{2s_2+2s}}\right )q^{2s}e^{s^2}
\frac{\dif s}{s}.
\end{split}
\end{align}
 
We shift the contour of integration in \eqref{lambdasumc=1} to the line $\Re(s)=\varepsilon$ and apply \eqref{Lsymest}, \eqref{Stirlingratiosquare} to bound the integral on the new line.  This reveals that \eqref{lambdasumc=1} is $\ll q^{\varepsilon}$. Similarly,
\begin{align*}
\begin{split}
  \sum_{\substack{m,n \\ (mn, q)=1 }}
 \frac{\lambda_f(m)\lambda_f(n)}{m^{1-s_1}n^{1-s_2}} \Wf_{1-s_1,1-s_2} \left(\frac {mn}{q^2}\right) \ll q^{\varepsilon}.
\end{split}
\end{align*}

Inserting these bounds in \eqref{lambdasumc=1'}, the total contribution of the case $c=1$ in the last display of \eqref{LNsquaresum1} is
\begin{align}
\label{contributionc=1}
\begin{split}
  \ll q^{\varepsilon}\Big (1+(|s_1|+1)^{1-2\sigma_1}(|s_2|+1)^{1-2\sigma_2}q^{2(1-\sigma_1-\sigma_2)}\Big ).
\end{split}
\end{align}

\subsection{Diagonal terms}  

 We consider first the terms $ma = nb$ (resp. $mb = na$) in the last expression of \eqref{LNsquaresum1}. As $(a, b)=1$, we may write $m = \alpha b, n =\alpha a$ (resp. $m = \alpha a, n =\alpha b$). Moreover, as $(ab, q)=1$, the condition $(mn, q)=1$ reduces to $(\alpha, q)=1$.  So \eqref{chistar} gives that these terms equal
\begin{align}
\label{mainterm}
\begin{split}
 & \frac{\phis(q)}{b^{s_1}a^{s_2}} F(s_1, s_2)+(\frac{q}{2\pi})^{2(1-s_1-s_2)}\frac {\Gamma(\frac{\kappa-1}{2}+1-s_1)\Gamma(\frac{\kappa-1}{2}+1-s_2)}{\Gamma(\frac{\kappa-1}{2}+s_1)\Gamma(\frac{\kappa-1}{2}+s_2)}
\frac{\phis(q)}{a^{1-s_1}b^{1-s_2}}F(1-s_1, 1-s_2),
\end{split}
\end{align}
  where
\begin{align*}
\begin{split}
 & F(s_1, s_2)=\sum_{(\alpha, q)=1} \frac{\lambda_f(\alpha b)\lambda_f(\alpha a)}{\alpha^{s_1+s_2}} \Wf_{s_1, s_2} \left(\frac{\alpha^2
ab}{q^2}\right).
\end{split}
\end{align*}

Writing $Y=q^2/(ab)$ for convenience and applying the definition of $\Wf_{s_1,s_2} (x)$ in Lemma \ref{PropDirpoly}, we arrive at
\begin{align}
\label{maintermint}
\begin{split}
 F(s_1, s_2)= & \frac{1}{2\pi i} \int\limits_{(2)}
\frac{\Gamma\left(\frac{\kappa-1}{2}+s_1 + s \right)\Gamma\left(\frac{\kappa-1}{2}+s_2 + s \right)}{(2\pi )^{2s} \Gamma\left(\frac{\kappa-1}{2}+s_1 \right)\Gamma\left(\frac{\kappa-1}{2}+s_2  \right) }L(2s+s_1+s_2, f \times f) H(2s+s_1+s_2; q, a, b)Y^{s}
\frac{\dif s}{s}.
\end{split}
\end{align}
  
 We then shift the integral in \eqref{maintermint} to the line of integration to $\Re(s)=1/4-(\sigma_1+\sigma_2)/2+\varepsilon$. Here we take $\varepsilon>0$ small enough so that $\Re(2s+s_1+s_2)>1/2$ on this new line.  As $s_1+s_2 \neq 1$, we deduce from \eqref{Lff} that we encounter simple poles at $s=0$ and $s=(1-s_1-s_2)/2$ (due to the simple pole of $\zeta(s)$ at $s=1$) in this process.  Let $\omega(n)$ denote the number of primes dividing $n$ and recall that $d(n)$ denotes the divisor function of $n$. By the well-known estimations (see \cite[Theorems 2.10, 2.11]{MVa1}) that for $n \geq 3$,
\begin{align}
\label{omegandnbound}
\omega(n), \log d(n) \ll \frac {\log n}{\log \log n}, 
\end{align}
 and the observation from \eqref{Hp} that $H_p(s; q, a, b)= \lambda_f(p^l)+ O(l/p^{\Re(s)})$ for $p^l \|  ab$, we infer that on the new line, for some constant $B_1$,
\begin{align*}
\begin{split}
 H(2s+s_1+s_2; q, a,b) \ll B_1^{\omega(q)+\omega(a)+\omega(b)}d(a)d(b) \ll (abq)^{\varepsilon}.
\end{split}
\end{align*}
 
  Combining the above with \eqref{Lsymest}, \eqref{zetaest} and the rapid decay of $\Gamma(s)$ as $|\Im(s)| \rightarrow \infty$, the integral on the new line is
\begin{align*}
\begin{split}
  \ll (\frac {q^2}{ab})^{1/4-(\sigma_1+\sigma_2)/2+\varepsilon}.
\end{split}
\end{align*}

Recall that the reside of $\zeta(s)$ at $s=1$ equals to $1$.  Taking account the residues at $s=0$ and $s=(1-s_1-s_2)/2$,
\begin{align}
\label{alphasum}
\begin{split}
  F(s_1, s_2)= \zeta(s_1+ &s_2)L(s_1+s_2, \operatorname{sym}^2 f)H(s_1+s_2; q, a,b)\\
&+\frac{\Gamma\left(\frac{\kappa+s_1-s_2}{2} \right)\Gamma\left(\frac{\kappa-s_1+s_2}{2} \right)L(1, \operatorname{sym}^2 f) H(1; q, a,b)}{(2\pi )^{1-s_1-s_2}(1-s_1-s_2) \Gamma\left(\frac{\kappa-1}{2}+s_1 \right)\Gamma\left(\frac{\kappa-1}{2}+s_2  \right) } \Big(\frac {q^2}{ab}\Big)^{(1-s_1-s_2)/2} \\
&+O\Big( \Big(\frac {ab}{q^2}\Big)^{1/4-(\sigma_1+\sigma_2)/2+\varepsilon}\Big). 
\end{split}
\end{align}
  Similarly,
\begin{align}
\label{alphasum1}
\begin{split}
  F(1-s_1, 1-s_2)= \zeta(2-s_1&-s_2)L(2-s_1-s_2, \operatorname{sym}^2 f)H(2-s_1-s_2; q, a,b)\\
&+\frac{\Gamma\left(\frac{\kappa+s_1-s_2}{2} \right)\Gamma\left(\frac{\kappa-s_1+s_2}{2} \right) L(1, \operatorname{sym}^2 f) H(1; q, a,b) }{(2\pi )^{s_1+s_2-1}(s_1+s_2-1) \Gamma\left(\frac{\kappa-1}{2}+1-s_1 \right)\Gamma\left(\frac{\kappa-1}{2}+1-s_2  \right) } \Big(\frac {q^2}{ab}\Big)^{(s_1+s_2-1)/2} \\
&+O\Big( \Big(\frac {q^2}{ab}\Big)^{-3/4+(\sigma_1+\sigma_2)/2+\varepsilon} \Big). 
\end{split}
\end{align}

From \eqref{mainterm}, \eqref{alphasum}, \eqref{alphasum1} and \eqref{Stirlingratiosquare1}, the terms $ma = nb$ and $mb = na$ in the last expression of \eqref{LNsquaresum1} are
\begin{align}
\label{mainterm1}
\begin{split}
 \frac{\phis (q)}{b^{s_1}a^{s_2}} & \zeta(s_1+s_2)L(s_1+s_2, \operatorname{sym}^2 f)H(s_1+s_2; q, a,b)\\
&+\frac{\phis(q)}{b^{s_1}a^{s_2}}\cdot \frac{\Gamma\left(\frac{\kappa+s_1-s_2}{2} \right)\Gamma\left(\frac{\kappa-s_1+s_2}{2} \right) L(1, \operatorname{sym}^2 f) H(1; q, a,b) }{(2\pi )^{1-s_1-s_2}(1-s_1-s_2) \Gamma\left(\frac{\kappa-1}{2}+s_1 \right)\Gamma\left(\frac{\kappa-1}{2}+s_2  \right) } \Big(\frac {q^2}{ab}\Big)^{(1-s_1-s_2)/2} \\
&+ \Big(\frac{q}{2\pi}\Big)^{2(1-s_1-s_2)}\frac {\Gamma(\frac{\kappa-1}{2}+1-s_1)\Gamma(\frac{\kappa-1}{2}+1-s_2)}{\Gamma(\frac{\kappa-1}{2}+s_1)\Gamma(\frac{\kappa-1}{2}+s_2)}
\frac{\phis(q)}{a^{1-s_1}b^{1-s_2}} \\
& \hspace*{2cm} \times \zeta(2-s_1-s_2)L(2-s_1-s_2, \operatorname{sym}^2 f)H(2-s_1-s_2; q, a,b)\\
&+\Big(\frac{q}{2\pi}\Big)^{2(1-s_1-s_2)}\frac{\phis(q)}{a^{1-s_1}b^{1-s_2}}\frac{\Gamma\left(\frac{\kappa+s_1-s_2}{2} \right)\Gamma\left(\frac{\kappa-s_1+s_2}{2} \right)L(1, \operatorname{sym}^2 f) H(1; q,a,b) }{(2\pi )^{s_1+s_2-1}(s_1+s_2-1) \Gamma\left(\frac{\kappa-1}{2}+s_1 \right)\Gamma\left(\frac{\kappa-1}{2}+s_2  \right) }\Big(\frac {q^2}{ab}\Big)^{(s_1+s_2-1)/2} \\
&+O\Big( \Big(\frac {q^2}{ab}\Big)^{1/4-(\sigma_1+\sigma_2)/2+\varepsilon} \frac{q}{b^{\sigma_1}a^{\sigma_2}}+(|s_1|+1)^{1-2\sigma_1}(|s_2|+1)^{1-2\sigma_2}\frac{q^{1+2(1-\sigma_1-\sigma_2)}}{a^{1-\sigma_1}b^{1-\sigma_2}}\Big(\frac {q^2}{ab}\Big)^{-3/4+(\sigma_1+\sigma_2)/2+\varepsilon}\Big) \\
= \frac{\phis(q)}{b^{s_1}a^{s_2}} & \zeta(s_1+s_2)L(s_1+s_2, \operatorname{sym}^2 f)H(s_1+s_2; q, a,b)\\
&+ \Big(\frac{q}{2\pi}\Big)^{2(1-s_1-s_2)}\frac {\Gamma(\frac{\kappa-1}{2}+1-s_1)\Gamma(\frac{\kappa-1}{2}+1-s_2)}{\Gamma(\frac{\kappa-1}{2}+s_1)\Gamma(\frac{\kappa-1}{2}+s_2)}
\frac{\phis(q)}{a^{1-s_1}b^{1-s_2}} \\
& \hspace*{2cm} \times \zeta(2-s_1-s_2)L(2-s_1-s_2, \operatorname{sym}^2 f)H(2-s_1-s_2; q, a,b)\\
&+O\Big( \Big(\frac {q^2}{ab}\Big)^{1/4-(\sigma_1+\sigma_2)/2+\varepsilon} \frac{q}{b^{\sigma_1}a^{\sigma_2}}+(|s_1|+1)^{1-2\sigma_1}(|s_2|+1)^{1-2\sigma_2}\frac{q^{1+2(1-\sigma_1-\sigma_2)}}{a^{1-\sigma_1}b^{1-\sigma_2}}\Big(\frac {q^2}{ab}\Big)^{-3/4+(\sigma_1+\sigma_2)/2+\varepsilon}\Big). 
\end{split}
\end{align}

\subsection{Off-diagonal terms}

It still remains to consider the contribution of the terms $ma \neq nb$ and $mb \neq na$ in the last expression of \eqref{LNsquaresum1}. Due to the rapid decay of $\Wf_{s_1,s_2}(x)$ (see \eqref{W}), we may assume that $mn \leq (|s_1|+1)^{1+\varepsilon}(|s_2|+1)^{1+\varepsilon}q^{2+\varepsilon}$.  Now \cite[Lemma 1.6]{BFKMM} gives that there exist two non-negative function ${\mathcal V}_1(x)$, ${\mathcal V}_2(x)$ supported on $[1/2,2]$,  satisfying
\begin{align}
\label{Vbounds}
\begin{split}
 {\mathcal V}^{(k)}_j(x) \ll_{k, \varepsilon} q^{k\varepsilon}.
\end{split}
\end{align}
  Moreover, we have the following smooth partition of unity:
\begin{align*}
\begin{split}
  \sum_{k \geq 0}{\mathcal V}_j \Big(\frac x{2^k}\Big)=1, \quad j=1,2.
\end{split}
\end{align*}

  Applying this, the definition of $\Wf_{s_1,s_2}(x)$ in Lemma \ref{PropDirpoly} and  \eqref{Stirlingratiosquare1}, the terms $ma \neq nb$ in the last expression of \eqref{LNsquaresum1} contribute
\begin{align}
\label{offdiagbounds}
\begin{split}
 \ll & \sum_{k_1, k_2} \sum_{c | q} \Big| \mu\Big(\frac{q}{c} \Big) \phi(c) \Big|\sum_{\substack{A=2^{k_1}, B=2^{k_2} \\ k_1, k_2 \geq 0 \\ AB \leq (|s_1|+1)^{1+\varepsilon}(|s_2|+1)^{1+\varepsilon}q^{2+\varepsilon}}} \frac {\mathcal{I}_{V_1, V_2} (s_1, s_2)}{A^{\sigma_1}B^{\sigma_2}}  \\
&+ (|s_1|+1)^{1-2\sigma_1}(|s_2|+1)^{1-2\sigma_2}q^{2(1-\sigma_1-\sigma_2)}\sum_{k_1, k_2}  \sum_{c | q} \Big| \mu\Big(\frac{q}{c} \Big) \phi(c) \Big| \sum_{\substack{A=2^{k_1}, B=2^{k_2} \\ k_1, k_2 \geq 0 \\ AB \leq (|s_1|+1)^{1+\varepsilon}(|s_2|+1)^{1+\varepsilon}q^{2+\varepsilon}}} \frac { \mathcal{I}_{V_3, V_4} (1-s_1, 1-s_2)}{A^{1-\sigma_1}B^{1-\sigma_2}} ,
\end{split}
\end{align}
 where
\begin{equation} \label{calIdef}
 \mathcal{I}_{V_j, V_l} (s_1, s_2) = \int\limits_{(\varepsilon)} \Big |
\sum_{\substack{ma \neq nb \\ (mn, q)=1 \\ ma \equiv nb  \bmod c}}
\lambda_f(m)\lambda_f(n)V_j \left(\frac {m}{A}\right) V_l \left(\frac {n}{B}\right)\Big (\frac {q^2}{AB}\Big )^{s} \frac{\Gamma\left(\frac{\kappa-1}{2} +s_1+ s \right)\Gamma\left(\frac{\kappa-1}{2} +s_2+ s \right)}{(2\pi)^{2s} \Gamma\left(\frac{\kappa-1}{2}+s_1 \right)\Gamma\left(\frac{\kappa-1}{2}+s_2 \right) }\Big | \frac{|e^{s^2} \dif s|}{|s|}
\end{equation}
 with
\begin{align*}
\begin{split}
V_j \left(x \right)= \left\{
 \begin{array}
  [c]{ll}
  x^{-s_j-s}{\mathcal V}_j(x), \quad j=1,2,\\
 x^{s_j-s-1}{\mathcal V}_{j-2}(x), \quad j=3,4 .
 \end{array}
 \right.
 \end{split} 
\end{align*}
Now the rapid decay of $e^{s^2}$ on the vertical line and \eqref{Stirlingratiosquare} ensure that truncating the two integrals defined in \eqref{calIdef} and appearing in \eqref{offdiagbounds} to $\Im(s) \leq (\log 5q)^2$ inccur an error of size 
\begin{align*}
\begin{split}
 \ll (|s_1|+1)^{\varepsilon}(|s_2|+1)^{\varepsilon}q^{-C},
\end{split}
\end{align*}
 for any constant $C$. \newline

  Observe further that when  $\Im(s) \leq (\log 5q)^2$, the bounds given in \eqref{Vbounds} are also satisfied by $V_i (i=1,2)$. Also,  note that the number of the effective summations over $k_1$ and $k_2$ is $O(\log ((|s_1|+4)(|s_2|+4))(\log q)^2)$. Moreover,
\begin{align*}
\begin{split}
\frac {1}{A^{\sigma_1}B^{\sigma_2}}&+\frac {(|s_1|+1)^{1-2\sigma_1}(|s_2|+1)^{1-2\sigma_2}q^{2(1-\sigma_1-\sigma_2)}}{A^{1-\sigma_1}B^{1-\sigma_2}} \\
\ll & \frac {(1+(|s_1|+1)^{1-2\sigma_1}(|s_2|+1)^{1-2\sigma_2}q^{2(1-\sigma_1-\sigma_2)})A^{|\sigma_1-1/2|}B^{|\sigma_2-1/2|}}{\sqrt{AB}} \\
\ll & (1+(|s_1|+1)^{1-2\sigma_1}(|s_2|+1)^{1-2\sigma_2}q^{2(1-\sigma_1-\sigma_2)})\big (|s_1|+1)^{1+\varepsilon}(|s_2|+1)^{1+\varepsilon}q^{2+\varepsilon}\big )^{|\sigma_1-1/2|+|\sigma_2-1/2|}\frac {1}{\sqrt{AB}}. 
\end{split}
\end{align*}
It follows from the above observations that the expression in \eqref{offdiagbounds} can be further bounded by
\begin{align}
\label{offdiagbounds1}
\begin{split}
 & (|s_1|+1)^{\varepsilon}(|s_2|+1)^{\varepsilon}q^{\varepsilon}(1+(|s_1|+1)^{1-2\sigma_1}(|s_2|+1)^{1-2\sigma_2}q^{2(1-\sigma_1-\sigma_2)})\big (|s_1|+1)(|s_2|+1)q^{2}\big )^{|\sigma_1-1/2|+|\sigma_2-1/2|} \\
& \times \max_{\substack{A, B \geq 1 \\AB \leq (|s_1|+1)^{1+\varepsilon}(|s_2|+1)^{1+\varepsilon}q^{2+\varepsilon}}} E(A,B) +(|s_1|+1)^{\varepsilon}(|s_2|+1)^{\varepsilon}q^{-C}(q+(|s_1|+1)^{1-2\sigma_1}(|s_2|+1)^{1-2\sigma_2}q^{1+2(1-\sigma_1-\sigma_2)}). 
\end{split}
\end{align}
  where 
\begin{align*}
\begin{split}
& E(A, B) \\
=& \frac {1}{\sqrt{AB}}\sum_{c | q}\Big| \mu\Big(\frac{q}{c} \Big) \phi(c) \Big|\Big (\Big |\sum_{\substack{ma \neq nb \\ (mn, q)=1 \\ ma \equiv nb \,\shortmod c}}
\lambda(m)\lambda(n) V_1 \left(\frac {m}{A}\right) V_2 \left(\frac {n}{B}\right) \Big |+\Big |\sum_{\substack{mb \neq na \\ (mn, q)=1 \\ mb \equiv na \,\shortmod c}}
\lambda(m)\lambda(n) V_3 \left(\frac {m}{A}\right) V_4 \left(\frac {n}{B}\right) \Big |\Big ).
\end{split}
\end{align*}

Now we must estimate $E(A, B)$ for integers $A, B \geq 1$, $AB \leq (|s_1|+1)^{1+\varepsilon}(|s_2|+1)^{1+\varepsilon}q^{2+\varepsilon}$ and functions $V_i, 1 \leq i \leq 4$ satisfying \eqref{Vbounds}. We notice that the estimation for $E(A, B)$ is the same as the one for $E(M, N)$ defined in \cite[(5.3)]{CH} with $k=1, X=1$ there. Without loss of generality, we may assume $A \leq B$ in the sequel.  The bounds in \eqref{lambdafbound} and \eqref{omegandnbound} trivially lead to
\begin{align*}
\begin{split}
& E(A, B) \ll \frac {q^{\varepsilon}}{\sqrt{AB}}\sum_{c | q}\Big| \mu\Big(\frac{q}{c} \Big) \phi(c) \Big| \Big (\sum_{\substack{ma \equiv nb \bmod c \\ A/2 \leq m \leq 2A, B/2 \leq n \leq 2B}}1+\sum_{\substack{mb \equiv na \bmod c \\ A/2 \leq m \leq 2A, B/2 \leq n \leq 2B}}1 \Big ) \ll q^{\varepsilon}(AB)^{1/2}.
\end{split}
\end{align*} 
The last estimate above emerges by noting that for a fixed $n$, there is only one $m \pmod c$ satisfying the condition $ma \equiv nb \,\pmod c$ or the condition $mb \equiv na \,\pmod c$. \newline

   It follows that we have 
\begin{align}
\label{Ebound1}
\begin{split}
 E(A, B) \ll q^{1-\eta/100+\varepsilon}, \quad AB < q^{2-\eta/25}. 
\end{split}
\end{align} 

  We may thus consider the case $q^{2-\eta/25} \leq AB  \leq (|s_1|+1)^{1+\varepsilon}(|s_2|+1)^{1+\varepsilon}q^{2+\varepsilon}$.  When $B<20A$, we have by \cite[(5.7)]{CH} that
\begin{align}
\label{Ebound2}
\begin{split}
& E(A, B) \ll q^{\varepsilon}\frac {B}{\sqrt{A}} \ll q^{\varepsilon}\sqrt{A} \ll (|s_1|+1)^{1/4+\varepsilon}(|s_2|+1)^{1/4+\varepsilon}q^{1/2+\varepsilon},
\end{split}
\end{align} 
  where the last inequality above follows as $A \leq B$, we have $A \leq (AB)^{1/2} \leq (|s_1|+1)^{1/2+\varepsilon}(|s_2|+1)^{1/2+\varepsilon}q^{1+\varepsilon}$. Thus, we may further assume that $q^{2-\eta/25}  \leq AB  \leq (|s_1|+1)^{1+\varepsilon}(|s_2|+1)^{1+\varepsilon}q^{2+\varepsilon}$, $B \geq 20A$. In which case, we apply \cite[(5.6)]{CH} to see that when $B/A < q^{1-\eta/5}$, 
\begin{align}
\label{Ebound2'}
\begin{split}
 E(A, B) \ll & q^{\varepsilon}\Big(\Big(\frac {B}{A}\Big)^{1/2}q^{1/2}+\Big(\frac {B}{A}\Big)^{1/2}(AB)^{1/4}+\Big(\frac {B}{A}\Big)^{1/4}q^{3/4}+\Big(\frac {B}{A}\Big)^{1/4}(AB)^{1/4}q^{1/4}) \\
\ll & q^{\varepsilon}\Big(\Big(\frac {B}{A}\Big)^{1/2}q^{1/2}+\Big(\frac {B}{A}\Big)^{1/4}q^{1/4}\Big) \ll q^{1-\eta+\varepsilon}.
\end{split}
\end{align}  

   We are now left with the case  $q^{2-\eta/25} \leq AB  \leq (|s_1|+1)^{1+\varepsilon}(|s_2|+1)^{1+\varepsilon}q^{2+\varepsilon}$, $B \geq 20A$ and that $B \geq q^{1-\eta/5}A$. When $q$ is large enough, the condition $B \geq q^{1-\eta/5}A$ implies that $B \geq 20A$. 
Thus, it remains to estimate $E(A, B)$ for integers $A, B \geq 1$ satisfying $q^{2-\eta/25} \leq AB  \leq (|s_1|+1)^{1+\varepsilon}(|s_2|+1)^{1+\varepsilon}q^{2+\varepsilon}, B \geq q^{1-\eta/5}A$.  If there exists a divisor $q_0$ of $q$ such that $q/q_0$ is odd and $q^{\eta} \ll q_0 \ll q^{1/2-\eta}$ for some $0<\eta <1/2$, \cite[(5.11)]{CH} allows us to deduce that
\begin{align}
\label{Ebound3}
\begin{split}
 E(A, B)  \ll & q^{\varepsilon}\Big((AB)^{1/4}\Big(\frac {B}{A}\Big)^{-1/4}q^{1/2}q^{1/2}_0+\Big(\frac {B}{A}\Big)^{-1/2}q^{5/4}q^{1/4}_0+(AB)^{-1/4}\Big(\frac {B}{A}\Big)^{-1/4}q^{7/4}q^{-1/4}_0) \\
\ll & q^{\varepsilon}((|s_1|+1)^{1/4+\varepsilon}(|s_2|+1)^{1/4+\varepsilon}q^{1/2}q^{-1/4+\eta/20}q^{1/2}q^{1/4-\eta/2} \\
& \hspace*{2cm} +q^{-1/2+\eta/10}q^{5/4}q^{1/8-\eta/4}+q^{-1/2+\eta/100}q^{-1/4+\eta/20}
q^{7/4}q^{-\eta/4}) \\
\ll & (|s_1|+1)^{1/4+\varepsilon}(|s_2|+1)^{1/4+\varepsilon}q^{1-19\eta/20+\varepsilon}+q^{7/8-3\eta/20+\varepsilon}+q^{1-19\eta/100+\varepsilon} \\
\ll & (|s_1|+1)^{1/4+\varepsilon}(|s_2|+1)^{1/4+\varepsilon}q^{1-19\eta/20+\varepsilon}+q^{1-19\eta/100+\varepsilon} .
\end{split}
\end{align} 

    We conclude from \eqref{scondition}, \eqref{offdiagbounds1}--\eqref{Ebound3} that by taking $C$ large enough, the expression in \eqref{offdiagbounds} is bounded by
\begin{align}
\label{offdiagboundsqcomposite}
\begin{split}
 (|s_1|&+1)^{\varepsilon}(|s_2|+1)^{\varepsilon}q^{\varepsilon}(1+(|s_1|+1)^{1-2\sigma_1}(|s_2|+1)^{1-2\sigma_2}q^{2(1-\sigma_1-\sigma_2)})\big (|s_1|+1)(|s_2|+1)q^{2}\big )^{|\sigma_1-1/2|+|\sigma_2-1/2|} \\
& \hspace*{2cm} \times \Big (q^{1-\eta/100+\varepsilon}+(|s_1|+1)^{1/4+\varepsilon}(|s_2|+1)^{1/4+\varepsilon}q^{1-19\eta/20+\varepsilon} \Big ) \\
& \hspace*{3cm} +(|s_1|+1)^{\varepsilon}(|s_2|+1)^{\varepsilon}q^{-C}(q+(|s_1|+1)^{1-2\sigma_1}(|s_2|+1)^{1-2\sigma_2}q^{1+2(1-\sigma_1-\sigma_2)}) \\
& \ll  q^{\varepsilon}(1+(|s_1|+1)^{1-2\sigma_1}(|s_2|+1)^{1-2\sigma_2}q^{2(1-\sigma_1-\sigma_2)})\big (|s_1|+1)(|s_2|+1)q^{2}\big )^{|\sigma_1-1/2|+|\sigma_2-1/2|}  \Big (q^{1-\eta/100}+q^{1-\varepsilon_0} \Big ). 
\end{split}
\end{align}

If $q$ is a prime, there are only two possible values of $c$ in the last display of \eqref{LNsquaresum1}:
$c=1$ and $c=q$. By \eqref{contributionc=1}, it remains to consider the case $c=q$ in the last display of \eqref{LNsquaresum1}. We note first that when $(mn,q)>1$, then either $q|m$ or $q|n$ but the conditions $ma \equiv nb \pmod q, mb \equiv na \pmod p$ and $(ab, q)=1$ then imply that $q|m$ and $q|n$ must hold simultaneously.  Therefore, by \eqref{lambdafbound} and \eqref{omegandnbound}, we see that removing the condition $(mn, q)=1$ in $E(A,B)$ leads to an error of size
\begin{align*}
\begin{split}
& \ll \frac {q^{1+\varepsilon}}{\sqrt{AB}}\sum_{\substack{m, n  \\ A/2 \leq mq \leq 2A, B/2 \leq nq \leq 2B}}1 \ll q^{-1+\varepsilon}(AB)^{1/2} \ll
(|s_1|+1)^{1/2+\varepsilon}(|s_2|+1)^{1/2+\varepsilon}q^{\varepsilon}.
\end{split}
\end{align*} 

  We thus derive from \eqref{contributionc=1},  \eqref{offdiagbounds1}--\eqref{Ebound2} and the above that \eqref{offdiagbounds} is majorized by
\begin{align}
\label{offdiagbounds2}
\begin{split}
 & (|s_1|+1)^{\varepsilon}(|s_2|+1)^{\varepsilon}q^{1+\varepsilon}(1+(|s_1|+1)^{1-2\sigma_1}(|s_2|+1)^{1-2\sigma_2}q^{2(1-\sigma_1-\sigma_2)})\big (|s_1|+1)(|s_2|+1)q^{2}\big )^{|\sigma_1-1/2|+|\sigma_2-1/2|} \\
& \hspace*{2cm} \times \max_{\substack{A, B \geq 1 \\ q^{2-\eta/25} \leq AB \leq (|s_1|+1)^{1+\varepsilon}(|s_2|+1)^{1+\varepsilon}q^{2+\varepsilon}\\ B \geq q^{1-\eta/5}A}} E'(A,B) \\
& + (|s_1|+1)^{\varepsilon}(|s_2|+1)^{\varepsilon}q^{\varepsilon}(1+(|s_1|+1)^{1-2\sigma_1}(|s_2|+1)^{1-2\sigma_2}q^{2(1-\sigma_1-\sigma_2)})\big (|s_1|+1)(|s_2|+1)q^{2}\big )^{|\sigma_1-1/2|+|\sigma_2-1/2|} \\
& \hspace*{2cm} \times \Big (q^{1-\eta/100+\varepsilon}+(|s_1|+1)^{1/4+\varepsilon}(|s_2|+1)^{1/4+\varepsilon}q^{1/2+\varepsilon} \Big ) \\
&+(|s_1|+1)^{\varepsilon}(|s_2|+1)^{\varepsilon}q^{-C}(q+(|s_1|+1)^{1-2\sigma_1}(|s_2|+1)^{1-2\sigma_2}q^{1+2(1-\sigma_1-\sigma_2)})\\
&+q^{\varepsilon}\Big (1+(|s_1|+1)^{1-2\sigma_1}(|s_2|+1)^{1-2\sigma_2}q^{2(1-\sigma_1-\sigma_2)}\Big )+(|s_1|+1)^{1/2+\varepsilon}(|s_2|+1)^{1/2+\varepsilon}q^{\varepsilon},
\end{split}
\end{align}
  where 
\begin{align} \label{E'def}
E'(A, B) = \frac {1}{\sqrt{AB}}\Big (\Big |\sum_{\substack{ma \neq nb  \\ ma \equiv nb \bmod q}}
\lambda(m)\lambda(n) V_1 \left(\frac {m}{A}\right) V_2 \left(\frac {n}{B}\right) \Big |+\Big |\sum_{\substack{mb \neq na  \\ mb \equiv na \bmod c}}
\lambda(m)\lambda(n) V_3 \left(\frac {m}{A}\right) V_4 \left(\frac {n}{B}\right) \Big |\Big ).
\end{align}

   Note that
\begin{align}
\label{offdiagbounds3}
\begin{split}
 &\max_{\substack{A, B \geq 1 \\ q^{2-\eta/25} \leq AB \leq (|s_1|+1)^{1+\varepsilon}(|s_2|+1)^{1+\varepsilon}q^{2+\varepsilon}\\ B \geq q^{1-\eta/5}A}} E'(A,B) \ll  \max_{\substack{A, B \geq 1 \\ q^{2-2\eta} \leq AB \leq (|s_1|+1)^{1+\varepsilon}(|s_2|+1)^{1+\varepsilon}q^{2+\varepsilon}\\ B \geq q^{1-4\eta}A}} E'(A,B)
\end{split}
\end{align}
 
   Moreover, due to similarities between its two constituent sums in \eqref{E'def}, 
\begin{align}
\label{Esimbound}
\begin{split}
 E'(A, B) \ll & \frac {1}{\sqrt{AB}}\Big |\sum_{\substack{ma \neq nb  \\ ma \equiv nb \bmod q}}
\lambda(m)\lambda(n) V_1 \left(\frac {m}{A}\right) V_2 \left(\frac {n}{B}\right) \Big |.
\end{split}
\end{align}

   Further, it follows from \cite[(3.3)]{BFKMM} and the paragraph below it together with the estimations given in \eqref{lambdafbound} and \eqref{omegandnbound} that
\begin{align}
\label{Eextraterm}
\begin{split}
\frac {1}{q\sqrt{AB}}\sum_{\substack{m, n}}
\lambda(m)\lambda(n) V_1 \left(\frac {m}{A}\right) V_2 \left(\frac {n}{B}\right) \ll & \frac {q^{\varepsilon}}{q\sqrt{AB}}.
\end{split}
\end{align}

   We now deduce from \eqref{offdiagbounds2}--\eqref{Eextraterm} that the expression in \eqref{offdiagbounds} is bounded by
\begin{align}
\label{offdiagbounds4}
\begin{split}
 & (|s_1|+1)^{\varepsilon}(|s_2|+1)^{\varepsilon}q^{1+\varepsilon}(1+(|s_1|+1)^{1-2\sigma_1}(|s_2|+1)^{1-2\sigma_2}q^{2(1-\sigma_1-\sigma_2)})\big (|s_1|+1)(|s_2|+1)q^{2}\big )^{|\sigma_1-1/2|+|\sigma_2-1/2|} \\
& \hspace*{2cm} \times \max_{\substack{A, B \geq 1 \\ q^{2-2\eta} \leq AB \leq (|s_1|+1)^{1+\varepsilon}(|s_2|+1)^{1+\varepsilon}q^{2+\varepsilon}\\ B \geq q^{1-4\eta}A}} \mathcal{E}(A, B) \\
& + (|s_1|+1)^{\varepsilon}(|s_2|+1)^{\varepsilon}q^{\varepsilon}(1+(|s_1|+1)^{1-2\sigma_1}(|s_2|+1)^{1-2\sigma_2}q^{2(1-\sigma_1-\sigma_2)})\big (|s_1|+1)(|s_2|+1)q^{2}\big )^{|\sigma_1-1/2|+|\sigma_2-1/2|} \\
& \hspace*{2cm} \times \Big (q^{1-\eta/100+\varepsilon}+(|s_1|+1)^{1/4+\varepsilon}(|s_2|+1)^{1/4+\varepsilon}q^{1/2+\varepsilon} \Big ) \\
&+(|s_1|+1)^{\varepsilon}(|s_2|+1)^{\varepsilon}q^{-C}(q+(|s_1|+1)^{1-2\sigma_1}(|s_2|+1)^{1-2\sigma_2}q^{1+2(1-\sigma_1-\sigma_2)})\\
&+q^{\varepsilon}\Big (1+(|s_1|+1)^{1-2\sigma_1}(|s_2|+1)^{1-2\sigma_2}q^{2(1-\sigma_1-\sigma_2)}\Big )+(|s_1|+1)^{1/2+\varepsilon}(|s_2|+1)^{1/2+\varepsilon}q^{\varepsilon},
\end{split}
\end{align}
  where
\begin{align*}
\begin{split}
& \mathcal{E}(A, B)=\frac {1}{\sqrt{AB}}\sum_{\substack{ma \neq nb \\ ma \equiv nb \bmod q}}
\lambda(m)\lambda(n) V_1 \left(\frac {m}{A}\right) V_2 \left(\frac {n}{B}\right)-\frac {1}{q\sqrt{AB}}\sum_{\substack{m, n}}
\lambda(m)\lambda(n) V_1 \left(\frac {m}{A}\right) V_2 \left(\frac {n}{B}\right).
\end{split}
\end{align*}
  
  We now estimate $\mathcal{E}(A,B)$ following the treatment in \cite[Section 6.2]{BFKMM} for the quantity $B^{\pm}_{f,g}(M,N)$ defined in \cite[(6.4)]{BFKMM}.
Note that the conditions  $q^{2-2\eta} \leq AB \leq (|s_1|+1)^{1+\varepsilon}(|s_2|+1)^{1+\varepsilon}q^{2+\varepsilon}, B \geq q^{1-4\eta}A$ imply that $B \geq q^{3/2-3\eta}$ and that we have $A \leq (AB)^{1/2} \leq (|s_1|+1)^{1/2+\varepsilon}(|s_2|+1)^{1/2+\varepsilon}q^{1+\varepsilon}$ as $A \leq B$.  It follows that upon taking $\varepsilon$ small enough, we can make the condition that $ma \neq nb$ vacuous if $(|s_1|+1)(|s_2|+1) \ll q^{1/4-\eta_0}$ and $a, b \leq q^{1/4}$.
  We further apply the additive characters to detect the condition $ma \equiv nb \, \pmod q$ to see that
\begin{align*}
\begin{split}
& \mathcal{E}(A, B)=\frac {1}{q\sqrt{AB}}\sum_{\substack{m, n}}
\lambda(m)\lambda(n) V_1 \left(\frac {m}{A}\right) V_2 \left(\frac {n}{B}\right)\sumstar_{c \bmod q}e\Big(\frac {(am-bn)c}{q} \Big ).
\end{split}
\end{align*}
  For any integer $(n, q)=1$, we denote $\overline{n}$ for a (fixed) integer satisfying $n\overline{n} \equiv 1 \pmod q$. We then apply the Voronoi summation formula given in \cite[Lemma 2.3]{BFKMM} to arrive at
\begin{align*}
\begin{split}
& \mathcal{E}(A, B)=\frac {1}{q\sqrt{AB^*}}\sum_{\substack{m, n}}
\lambda(m)\lambda(n) V_1 \left(\frac {m}{A}\right) \frac{1}{B} \widetilde{V}_{2,B} \left(\frac {n}{q^2}\right)
\sumstar_{c \shortmod q}e\Big(\frac {amc + n \overline{bc}}{q} \Big ),
\end{split}
\end{align*}
  where $B^*=q^2/B$ and
\begin{align*}
\begin{split}
 \widetilde{V}_{2,B}(y) = \int\limits^{\infty}_0V_2\Big(\frac xB\Big)\mathcal{J}(4\pi\sqrt{xy}) \dif x, \quad \mbox{with} \quad \mathcal{J}(x)=2\pi i^{\kappa}J_{\kappa-1}(x).
\end{split}
\end{align*}
  Here $J_{\kappa-1}(x)$ is the $J$-Bessel function. \newline

It is shown in \cite[Lemma 2.4]{BFKMM} that the functions $y \mapsto \widetilde{V}_{2,B} \left(y/q^2 \right)/B$ decays rapidly for $y \geq q^{\varepsilon}B^*$ so that we may assume further that $n \leq q^{\varepsilon}B^* <q$. Thus $n$ is invertible modulo $q$ and we can recast $\mathcal{E}(A,B)$ as
\begin{align*}
\begin{split}
& \mathcal{E}(A, B)=\frac {1}{q\sqrt{AB^*}} \sum_{\substack{m, n}}
\lambda(m)\lambda(n) V_1 \left(\frac {m}{A}\right) \frac{1}{B} \widetilde{V}_{2,B} \left(\frac {n}{q^2}\right) S(abm\overline{n}, 1, q),
\end{split}
\end{align*}
  where $S$ is the Kloosterman sum defined by
\begin{align*}
  S(u,v, q)=\sumstar_{\substack{ h \shortmod q }} e \Big(\frac {uh+v\overline h}{q}\Big).
\end{align*}

By virtue of the well-known Weil's bound for the Kloosterman sum given as in \cite[Corollary 11.12]{iwakow}, we see that $|S(abm\overline{n}, 1, q)| \leq 2q^{1/2}$. It follows from the above and the Cauchy-Schwarz
inequality that we have
\begin{align*}
\begin{split}
& \mathcal{E}(A, B) \ll  \frac {1}{\sqrt{qAB^*}} \sum_{\substack{m \ll A, n \ll q^{\varepsilon}B^*}}
|\lambda(m)\lambda(n)| \ll \frac {1}{\sqrt{qAB^*}}\Big(\sum_{\substack{m \ll A, n \ll q^{\varepsilon}B^*}}|\lambda(m)|^2\Big)^{1/2}\Big(\sum_{\substack{m \ll A, n \ll q^{\varepsilon}B^*}}|\lambda(n)|^2\Big)^{1/2}.
\end{split}
\end{align*}

   Note that by \cite[(2.4)]{BFKMM}, we have for any $x \geq 1$ and any $\varepsilon>0$,
\begin{align*}
\begin{split}
 \sum_{\substack{n \leq x}}|\lambda(n)|^2 \ll x^{1+\varepsilon}.
\end{split}
\end{align*}

   It follows from this that when $A/B < q^{-1-2\eta}$, we have
\begin{align}
\label{mathcalEbound1}
\begin{split}
& \mathcal{E}(A, B) \ll q^{-1/2+\varepsilon}\sqrt{AB^*}\ll  q^{-\eta+\varepsilon}. 
\end{split}
\end{align}
   It therefore remains to consider the case $A/B \geq q^{-1-2\eta}$ which is equivalent to $AB^* \geq q^{1-2\eta}$. Note further that the condition $B/A  \geq q^{1-4\eta}$ is equivalent to $AB^* \leq q^{1+4\eta}$. 
  Moreover, the condition $q^{2-2\eta}\leq AB  \leq (|s_1|+1)^{1+\varepsilon}(|s_2|+1)^{1+\varepsilon}q^{2+\varepsilon}$ implies
that $q^{-2\eta}< A/B^* \leq (|s_1|+1)^{1+\varepsilon}(|s_2|+1)^{1+\varepsilon}q^{\varepsilon}$. Together with the condition that $AB^* \geq q^{1-2\eta}$, this further implies that $q^{1/2-2\eta}\leq A$.  In this case we apply \cite[Proposition 5.5]{BFKMM} (note that this proposition originally assumes
\cite[Conjecture 5.7]{BFKMM} and is fully established in \cite[Theorem 1.1]{KMS17}) to arrive at in this case we also have
\begin{align}
\label{mathcalEbound2}
\begin{split}
& \mathcal{E}(A, B) \ll  q^{-1/2+\varepsilon}\sqrt{AB^*}(A^{-1/2}+q^{11/64}(AB^*)^{-3/16}) \ll  q^{-\eta+\varepsilon},
\end{split}
\end{align}
  provided that we have $A \leq q^{1/4}B^*$. As we have $(|s_1|+1)(|s_2|+1)\ll q^{1/4-\varepsilon_0}$, we see that this condition is implied by the condition that 
$A/B^* \leq (|s_1|+1)^{1+\varepsilon}(|s_2|+1)^{1+\varepsilon}q^{\varepsilon}$. Thus the estimation obtained in  is valid. \newline

  We conclude from \eqref{scondition}, \eqref{offdiagbounds4}--\eqref{mathcalEbound2} that by taking $C$ large enough, the expression in \eqref{offdiagbounds} is bounded by
\begin{align}
\label{offdiagbounds5}
\begin{split}
 & (|s_1|+1)^{\varepsilon}(|s_2|+1)^{\varepsilon}q^{1-\eta+\varepsilon}(1+(|s_1|+1)^{1-2\sigma_1}(|s_2|+1)^{1-2\sigma_2}q^{2(1-\sigma_1-\sigma_2)})\big (|s_1|+1)(|s_2|+1)q^{2}\big )^{|\sigma_1-1/2|+|\sigma_2-1/2|} \\
& + (|s_1|+1)^{\varepsilon}(|s_2|+1)^{\varepsilon}q^{\varepsilon}(1+(|s_1|+1)^{1-2\sigma_1}(|s_2|+1)^{1-2\sigma_2}q^{2(1-\sigma_1-\sigma_2)})\big (|s_1|+1)(|s_2|+1)q^{2}\big )^{|\sigma_1-1/2|+|\sigma_2-1/2|} \\
& \hspace*{2cm} \times \Big (q^{1-\eta/100+\varepsilon}+(|s_1|+1)^{1/4+\varepsilon}(|s_2|+1)^{1/4+\varepsilon}q^{1/2+\varepsilon} \Big ) \\
&+(|s_1|+1)^{\varepsilon}(|s_2|+1)^{\varepsilon}q^{-C}(q+(|s_1|+1)^{1-2\sigma_1}(|s_2|+1)^{1-2\sigma_2}q^{1+2(1-\sigma_1-\sigma_2)})\\
&+q^{\varepsilon}\Big (1+(|s_1|+1)^{1-2\sigma_1}(|s_2|+1)^{1-2\sigma_2}q^{2(1-\sigma_1-\sigma_2)}\Big )+(|s_1|+1)^{1/2+\varepsilon}(|s_2|+1)^{1/2+\varepsilon}q^{\varepsilon} \\
& \hspace*{1cm} \ll  q^{1-\eta/100+\varepsilon}(1+(|s_1|+1)^{1-2\sigma_1}(|s_2|+1)^{1-2\sigma_2}q^{2(1-\sigma_1-\sigma_2)})\big (|s_1|+1)(|s_2|+1)q^{2}\big )^{|\sigma_1-1/2|+|\sigma_2-1/2|}. 
\end{split}
\end{align}

\subsection{Conclusion}

  We now deduce the expression in \eqref{twistedsecondmoment} for case i) from \eqref{mainterm1} and \eqref{offdiagboundsqcomposite}. We also deduce the expression in \eqref{twistedsecondmoment} for case ii) from \eqref{mainterm1} and \eqref{offdiagbounds5}.

\section{Proofs of Theorem \ref{thmlowerbound}--\ref{thmupperbound}}
\label{sec 2'}

\subsection{Initial Treatments}
   
 As the case $k=0$ is trivial and the case $k=1$ follows from Corollary \ref{cortwistedsecondmomenttspecial} by setting $a=b=1$ there, we consider only the case $0< k \neq 1$ in what follows.  Let $N, M$ be two large natural numbers depending on $k$ only and $\{ \ell_j \}_{1 \leq j \leq R}$ a sequence of even natural
  numbers with $\ell_1= 2\lceil N \log \log q\rceil$ and $\ell_{j+1} = 2 \lceil N \log \ell_j \rceil$ for $j \geq 1$, where $R$ is the largest natural number such that $\ell_R >10^M$. \newline

Write further ${ P}_1$  the set of odd primes not exceeding $q^{1/\ell_1^2}$ and
${P_j}$ the set of primes lying in the interval $(q^{1/\ell_{j-1}^2}, q^{1/\ell_j^2}]$ for $2\le j\le R$. Define for each $1 \leq j \leq R$, 
\begin{equation*}
{\mathcal P}_j(t, \chi) = \sum_{p\in P_j} \frac{\lambda_f(p)}{p^{1/2+it}} \chi(p) \quad  \mbox{and} \quad {\mathcal Q}_j(t,\chi, k) =\Big (\frac{c_k {\mathcal
P}_j(t, \chi) }{\ell_j}\Big)^{r_k\ell_j},
\end{equation*}
  where
\begin{align*}
  c_k=  64 \max (1, k) \quad \mbox{and} \quad 
r_k = \left\{
 \begin{array}
  [c]{ll}
   2 & k>1,\\
  \lceil 1+1/k \rceil+1 & k<1.
 \end{array}
 \right.
\end{align*}
  We also set ${\mathcal Q}_{R+1}(t, \chi, k)=1$. \newline

Furthermore, we define for each $1 \leq j \leq R$ and any real number $\alpha$,
\begin{align*}
{\mathcal N}_j(t, \chi, \alpha) = E_{\ell_j} (\alpha {\mathcal P}_j(t,\chi)) \quad \mbox{and} \quad  \mathcal{N}(t, \chi, \alpha) = \prod_{j=1}^{R} {\mathcal
N}_j(t,\chi,\alpha),
\end{align*}
   where, for any non-negative integer $\ell$ and any real number $x$, 
\begin{equation*}
E_{\ell}(x) = \sum_{j=0}^{\ell} \frac{x^{j}}{j!}.
\end{equation*}

 In what follows, we follow the convention that an empty product equals $1$.  Also, in the remainder of the paper, the implied constants in $\ll$ or the $O$-symbol  depend on $k$ only.  \newline

  Now, arguing as in the proofs of \cite[Lemma 3.1-3.2]{GHH} by applying the lower bounds principle of W. Heap and K. Soundararajanand in \cite{H&Sound} and the
  upper bounds principle of M. Radziwi{\l\l} and K. Soundararajan in \cite{Radziwill&Sound}, we arrive at the following analogues of \cite[Lemmas 3.1, 3.2]{GHH}. 
\begin{lemma}
\label{lem1}
 With notations as above, for $0<k<1$,
\begin{align}
\label{basiclowerbound}
\begin{split}
\sumstar_{\substack{ \chi \shortmod q }} & L(\tfrac{1}{2}+it,f \otimes \chi)  \mathcal{N}(t, \chi, k-1) \mathcal{N}(-t, \overline{\chi}, k) \\
&  \ll \Big ( \sumstar_{\substack{ \chi \shortmod q }}|L(\tfrac{1}{2}+it, f \otimes \chi)|^{2k} \Big )^{1/2}\Big ( \sumstar_{\substack{ \chi \shortmod q
 }}|L(\tfrac{1}{2}+it, f \otimes  \chi)|^2 |\mathcal{N}(t, \chi, k-1)|^2  \Big)^{(1-k)/2} \\
 & \hspace*{2cm} \times \Big ( \sumstar_{\substack{ \chi \shortmod q }}   \prod^R_{j=1}\big ( |{\mathcal N}_j(t, \chi, k)|^2+ |{\mathcal Q}_j(t, \chi,k)|^2 \big )
 \Big)^{k/2}.
\end{split}
\end{align}
 For $k>1$,
\begin{align}
\label{basicboundkbig}
\begin{split}
\sumstar_{\substack{ \chi \shortmod q }} & L(\tfrac{1}{2}+it, f \otimes \chi)  \mathcal{N}(t, \chi, k-1) \mathcal{N}(-t,\overline{\chi}, k) \\
 & \ll  \Big ( \sumstar_{\substack{ \chi \shortmod q }}|L(\tfrac{1}{2}+it,f \otimes  \chi)|^{2k} \Big )^{1/2k}\Big ( \sumstar_{\substack{ \chi
 \shortmod q }} \prod^R_{j=1} \big ( |{\mathcal N}_j(t, \chi, k)|^2+ |{\mathcal Q}_j(t, \chi,k)|^2 \big ) \Big)^{(2k-1)/(2k)}.
\end{split}
\end{align}
\end{lemma}

\begin{lemma}
\label{lem2}
 With notations as above, for $0<k<1$, we have
\begin{align*}
\begin{split}
\sumstar_{\substack{ \chi \shortmod q }} & |L(\tfrac{1}{2}+it,f \otimes \chi)|^{2k} \\
 & \ll \Big ( \sumstar_{\substack{ \chi \shortmod q }}|L(\tfrac{1}{2}+it,f \otimes \chi)|^2 \sum^{R}_{v=0} \Big (\prod^v_{j=1} |\mathcal{N}_j(t, \chi, k-1)|^2 \Big ) |{\mathcal
 Q}_{v+1}(t, \chi, k)|^{2}
 \Big)^{k} \\
 & \hspace*{2cm}  \times \Big ( \sumstar_{\substack{ \chi \shortmod q }}  \sum^{R}_{v=0}  \Big (\prod^v_{j=1}|\mathcal{N}_j(t, \chi, k)|^2\Big )|{\mathcal
 Q}_{v+1}(t, \chi, k)|^{2} \Big)^{1-k}.
\end{split}
\end{align*}
\end{lemma}

Hence from Lemmas \ref{lem1} and \ref{lem2}, in order to prove Theorem \ref{thmlowerbound} and Theorem
   \ref{thmupperbound}, it suffices to establish the following three propositions.
\begin{proposition}
\label{Prop4} With the notation as above, for $k>0$,
\begin{align*}
\sumstar_{\substack{ \chi \shortmod q }}L(\tfrac{1}{2}+it, f \otimes \chi) \mathcal{N}(t, \overline{\chi}, k) \mathcal{N}(t, \chi, k-1) \gg \phis(q)(\log q)^{ k^2
} .
\end{align*}
\end{proposition}

\begin{proposition}
\label{Prop5} With the notation as above, for $0<k < 1$,
\begin{align*}
\max \Big (  \sumstar_{\substack{ \chi \shortmod q }} & |L(\tfrac{1}{2}+it,f \otimes \chi)\mathcal{N}(t, \chi, k-1)|^2, \\
& \sumstar_{\substack{ \chi \shortmod q }}|L(\tfrac{1}{2}+it,f \otimes \chi)|^2 \sum^{R}_{v=0}\Big (\prod^v_{j=1}|\mathcal{N}_j(t, \chi, k-1)|^{2}\Big ) |{\mathcal
 Q}_{v+1}(t, \chi, k)|^2 \Big ) \ll \phis(q)(\log q)^{ k^2 }.
\end{align*}
\end{proposition}

\begin{proposition}
\label{Prop6} With the notation as above, for $k>0$,
\begin{align*}
\max \Big ( \sumstar_{\substack{ \chi \shortmod q }}\prod^R_{j=1}\big ( |{\mathcal N}_j(t,\chi, k)|^2+ |{\mathcal Q}_j(t,\chi,k)|^2 \big ),  \sumstar_{\substack{ \chi \shortmod q }} \sum^{R}_{v=0} \Big ( \prod^v_{j=1}|\mathcal{N}_j(t,\chi, k)|^{2}\Big )|{\mathcal
 Q}_{v+1}(t,\chi, k)|^2 \Big )   \ll \phis(q)(\log q)^{ k^2 }.
\end{align*}
\end{proposition}

As the proof of Proposition \ref{Prop6} is similar to that of \cite[Proposition 3.5]{GHH}, we omit it here and focus on Propositions \ref{Prop4} and \ref{Prop5} in what follows. 

\subsection{Proof of Proposition \ref{Prop4}}
\label{sec 4}

We proceed in a way similar to the proof of \cite[Proposition 3.3]{GHH}.  Upon taking $M$ large enough, we may write for simplicity that
\begin{align*}
 {\mathcal N}(t, \chi, k-1)= \sum_{a  \leq q^{2/10^{M}}} \frac{x_a}{a^{1/2+it}} \chi(a) \quad \mbox{and} \quad \mathcal{N}(-t, \overline{\chi}, k) = \sum_{b  \leq
 q^{2/10^{M}}} \frac{y_b}{b^{1/2-it}}\overline{\chi}(b),
\end{align*}
    where for any $\varepsilon>0$,
\begin{align}
\label{xybounds}
 x_a, y_b \ll q^{\varepsilon}.
\end{align}
 From, the approximate functional equation given in Lemma \ref{PropDirpoly}, emerges the equality
\begin{align}
\label{Lfirstmoment}
\begin{split}
\sumstar_{\substack{ \chi \shortmod q }} & L(\tfrac{1}{2}+it,f \otimes \chi) \mathcal{N}(t, \chi, k-1)\mathcal{N}(-t,\overline{\chi}, k) \\
&=  \sumstar_{\substack{ \chi \shortmod q }}\sum_{m}\frac {\lambda_f(m)\chi(m)}{m^{1/2+it}}\mathcal{N}(t, \chi,
k-1)\mathcal{N}(-t\overline{\chi}, k)  W_t\left(\frac {mX}{q}\right) \\
& \hspace*{2cm} +\sumstar_{\substack{ \chi \shortmod q }}\iota_{\chi}\frac {(2\pi)^{2it}}{q^{2it}}\frac {\Gamma(\frac{\kappa}{2}-it)}{\Gamma(\frac{\kappa}{2}+it)}\sum_{m}\frac {\lambda_f(m)\overline\chi(m)}{m^{1/2-it}}\mathcal{N}(t, \chi,
k-1)\mathcal{N}(-t,\overline{\chi}, k)  W_{-t}\left(\frac {m}{qX}\right).
\end{split}
\end{align}

Now \eqref{sumchistar} gives that the right-hand side of \eqref{Lfirstmoment} is equal to
\begin{align*}
 \sum_{c | q}\mu\Big( \frac{q}{c} \Big) \phi(c) &  \sum_{(a,q)=1} \sum_{(b,q)=1}  \sum_{\substack{\substack{ (m,q)=1 \\ am \equiv b \bmod c}}}\frac {\lambda_f(m) x_a y_b}{(am)^{1/2+it}b^{1/2-it}}W_t\left(\frac {mX}{q}\right) \\
 & +\frac {(2\pi)^{2it}}{q^{2it}}\frac {\Gamma(\frac{\kappa}{2}-it)}{\Gamma(\frac{\kappa}{2}+it)}\sum_{a} \sum_{b} \sum_{\substack{m}}\frac {\lambda_f(m)x_ay_b}{a^{1/2+it}(bm)^{1/2-it}}W_{-t}\left(\frac {m}{qX}\right)\sumstar_{\substack{ \chi \shortmod q }}\iota_{\chi}\chi(a)\overline \chi(mb).
\end{align*}

  Note that as shown in the proof of \cite[Proposition 3.3]{GHH},
\begin{align*}
 \sumstar_{\substack{ \chi \shortmod q }}\iota_{\chi}\chi(a)\overline \chi(mb) \leq \frac {1}{q}\sum_{c | q}\phi(c)\frac {q(a,c)}{c} d(q)q^{1/2} \ll aq^{\half+\varepsilon}.
\end{align*}

Now Stirling's formula (see \cite[(5.113)]{iwakow}) reveals that
\begin{align*}
  \frac {\Gamma(\frac{\kappa}{2}-it)}{\Gamma(\frac{\kappa}{2}+it)} \ll 1.
\end{align*}

Utilizing the above, we get
\begin{align}
\label{Firstmomentsum2}
\begin{split}
\frac {(2\pi)^{2it}}{q^{2it}}\frac {\Gamma(\frac{\kappa}{2}-it)}{\Gamma(\frac{\kappa}{2}+it)} & \sum_{a} \sum_{b} \sum_{\substack{m}}\frac {\lambda_f(m)x_ay_b}{a^{1/2+it}(bm)^{1/2-it}}W_{-t}\left(\frac {m}{qX}\right)\sumstar_{\substack{ \chi \shortmod q }}\iota_{\chi}\chi(a)\overline \chi(mb). \\
& \ll q^{1/2+\varepsilon}X^{\varepsilon} \sum_{a} \sum_{b} \sum_{\substack{m}}\frac {a}{\sqrt{abm}}W_{-t}\left(\frac {m}{qX}\right)  \ll (|t|+1)^{1/2+\varepsilon}q^{4/10^{M}}q^{1/2+\varepsilon}X^{\varepsilon}\sqrt{qX},
\end{split}
\end{align}
  where the last bound follows by noting that due to the rapid decay of $\Wf_{-t}(x)$ given in \eqref{W}, we may take that $m \leq (qX(|t|+1))^{1+\varepsilon}$ in the summations over $m$ above. \newline

  It remains to evaluate
\begin{align}
\label{Firstmomentsum1}
 \sum_{c | q}\mu \Big( \frac{q}{c} \Big) \phi(c) \sum_{(a,q)=1} \sum_{(b,q)=1}  \sum_{\substack{\substack{ (m,q)=1 \\ am \equiv b \bmod c}}}\frac {\lambda_f(m) x_a y_b}{(am)^{1/2+it}b^{1/2-it}}W_t\left(\frac {mX}{q}\right) .
\end{align}

  We first consider the contribution from the terms $am=b+l c$ with $l \geq 1$ above. 
  Again the rapid decay of $W_t(x)$ (see \eqref{W}) enables us to restrict $m$ to $m \leq  (|t|+1)^{1+\varepsilon}(q/X)^{1+\varepsilon}$ and
this translates to $l \leq  (|t|+1)^{1+\varepsilon}q^{1+2/10^{M}+\varepsilon}/(Xc)$.  Also note that $am \geq lc$ so that we deduce together with \eqref{xybounds} that the total contribution from these terms is
\begin{align}
\label{Firstmomentsum1nondiag1}
 \ll &  \sum_{c | q} \phi(c)  q^{\varepsilon}X^{\varepsilon} \sum_{b  \leq q^{2/10^{M}}}  \sum_{l \leq (|t|+1)^{1+\varepsilon}q^{1+2/10^{M}+\varepsilon}/(Xc)}\frac {d(b+l c)}{\sqrt{blc}}
\ll (|t|+1)^{1/2+\varepsilon}X^{-1/2+\varepsilon}q^{1/2+2/10^{M}+\varepsilon}.
\end{align}

  Similarly, the contribution from the terms $b=am+l c$ with $l \geq 1$ in \eqref{Firstmomentsum1} is (by noting that
$l \leq  q^{2/10^{M}}/c$ and $b \geq lc$)
\begin{align}
\label{Firstmomentsum1nondiag2}
 \ll &  \sum_{c | q} \phi(c)  q^{\varepsilon}X^{\varepsilon} \sum_{a  \leq q^{2/10^{M}}}\sum_{m \leq (|t|+1)^{1+\varepsilon}(q/X)^{1+\varepsilon}}
  \sum_{l \leq q^{2/10^{M}}/c}\frac {1}{\sqrt{amlc}}
\ll (|t|+1)^{1/2+\varepsilon}X^{-1/2+\varepsilon}q^{1/2+2/10^{M}+\varepsilon}.
\end{align}

Now setting $X=q^{-1/2}$ to see from \eqref{Lfirstmoment}, \eqref{Firstmomentsum2}, \eqref{Firstmomentsum1nondiag1} and \eqref{Firstmomentsum1nondiag2}, we arrive at
\begin{equation} \label{Aestmation}
\begin{split}
\sumstar_{\substack{ \chi \shortmod q }} L(\tfrac{1}{2}+it, & f \otimes \chi) \mathcal{N}(t, \chi, k-1)\mathcal{N}(-t, \overline{\chi}, k)  \\
\gg & \phis(q)  \sum_{(a,q)=1} \sum_{(b,q)=1}  \sum_{\substack{ (m,q)=1 \\ m \leq (|t|+1)^{1+\varepsilon}(q/X)^{1+\varepsilon} \\ am = b }}
\frac {\lambda_f(m) x_a y_b}{\sqrt{abm}}+O((|t|+1)^{1/2+\varepsilon}q^{3/4+\varepsilon})\\
=& \phis(q) \sum_{(b,q)=1}  \frac {y_b}{b} \sum_{\substack{a, m \\ am = b }}\lambda_f(m)x_a+O((|t|+1)^{1/2+\varepsilon}q^{3/4+\varepsilon}),
\end{split}
\end{equation}
where the last equality above follows from the observation that $b \leq q^{2/10^{M}}<(|t|+1)^{1+\varepsilon}(q/X)^{1+\varepsilon}$. \newline

Note that the first term in the last expression of \eqref{Aestmation} is independent of $t$.  Proceeding as in the proof of \cite[Proposition 3.3]{GHH},  the last expression in \eqref{Aestmation} is $\gg \phis(q)(\log q)^{k^2}$ for $|t| \leq q^{1/8-\varepsilon_0}$.  This completes the proof of the proposition.

\subsection{Proof of Proposition \ref{Prop5}}
\label{sec 5}

  As the arguments are analogue, it suffices to show that
\begin{align}
\label{sumLsquareNQ} 
\sum^{R}_{v=0} \ \sumstar_{\substack{ \chi \shortmod q }}|L(\tfrac{1}{2}+it,f \otimes \chi)|^2 \Big (\prod^v_{j=1}|\mathcal{N}_j(t, \chi, k-1)|^{2}\Big ) |{\mathcal
 Q}_{v+1}(t, \chi, k)|^2  \ll \phis(q)(\log q)^{ k^2 }.
\end{align}

  We then argue as in the proof of \cite[Proposition 3.4]{GHH} to see that we may write for simplicity 
\begin{align}
\label{Lsquarepoly} 
 \Big (\prod^v_{j=1}|\mathcal{N}_j(t, \chi, k-1)|^2 \Big )|{\mathcal
 Q}_{v+1}(t,\chi, k)|^{2} = \Big( \frac{c_k  }{\ell_{v+1}}\Big)^{2r_k\ell_{v+1}}((r_k\ell_{v+1})!)^2 \sum_{\substack{a,b \leq q^{2r_k/10^{M}} \\ (ab, q)=1}} \frac{u_a u_b}{a^{1/2+it}b^{1/2-it}}\chi(a)\overline{\chi}(b),
\end{align}
 where
\begin{align}
\label{ubounds}
 \Big( \frac {c_k }{\ell_{v+1}}\Big)^{2r_k\ell_{v+1}}((r_k\ell_{v+1})!)^2, u_a, u_b \ll q^{\varepsilon}.
\end{align}
 Here we note that we may restrict the sums over $a,b$ in \eqref{Lsquarepoly} to be over those satisfying $(ab, q)=1$ for otherwise we have $\chi(a)=0$ or $\chi(b)=0$. \newline

   Upon writing $a=(a,b)\cdot a/(a,b), b=(a,b)\cdot b/(a,b)$, we see that $\chi(a) \overline{\chi}(b)=\chi(a/(a,b)) \overline{\chi}(b/(a,b))$.  Note that $(a/(a,b), b/(a,b))=1$.  We further take $M$ large enough so that $a, b \leq q^{2r_k/10^{M}} \leq q^{1/4}$.  We are therefore able to apply \eqref{twistedsecondmomenttspecial} to evaluate the inner sum on the right-hand side of \eqref{Lsquarepoly}.  This leads to
\begin{align}
\label{LNsquaresum}
\begin{split}
\sumstar_{\chi \shortmod q} |L&(\half + it,  f \otimes \chi)|^2\Big (\prod^v_{j=1}|\mathcal{N}_j(t, \chi, k-1)|^2 \Big )|{\mathcal
 Q}_{v+1}(t, \chi, k)|^{2} \\
&= \Big( \frac{c_k  }{\ell_{v+1}}\Big)^{2r_k\ell_{v+1}}((r_k\ell_{v+1})!)^2 \sum_{\substack{a,b \leq q^{2r_k/10^{M}} \\ (ab, q)=1}}\frac{u_a u_b}{a^{1/2+it}b^{1/2-it}} {\sumstar_{\chi \shortmod q}} |L(\half+it, f \otimes \chi)|^2\chi\Big( \frac {a}{(a,b)} \Big) \overline{\chi} \Big(\frac {b}{(a,b)}\Big) \\
&=   \phis(q)\Big( \frac {c_k }{\ell_{v+1}}\Big)^{2r_k\ell_{v+1}}((r_k\ell_{v+1})!)^2\sum_{\substack{a,b \leq q^{2r_k/10^{M}} \\ (ab, q)=1}}\frac{u_a u_b(a,b)}{\sqrt{ab}} L(1, \operatorname{sym}^2 f)H\Big(1; q, \frac {a}{(a,b)},\frac {b}{(a,b)}\Big) \\
& \hspace*{1cm} \times \Big (2\log (\frac{q}{2\pi})+2L'(1, \operatorname{sym}^2 f)+2H'\Big(1; q, \frac {a}{(a,b)},\frac {b}{(a,b)}\Big) \\
& \hspace*{4cm} + \frac {\Gamma'(\frac{\kappa}{2}+it_1)}{\Gamma(\frac{\kappa}{2}+it_1)}+\frac {\Gamma'(\frac{\kappa}{2}-it_1)}{\Gamma(\frac{\kappa}{2}-it_1)}-\log \Big(\frac {ab}{(a,b)^2}\Big) \Big )\\
& \hspace*{1cm} +O\Big( \Big( \frac{c_k  }{\ell_{v+1}}\Big)^{2r_k\ell_{v+1}}((r_k\ell_{v+1})!)^2 \sum_{\substack{a,b \leq q^{2r_k/10^{M}} \\ (ab, q)=1}}\frac{u_a u_b}{\sqrt{ab}} \Big( \Big(\frac {q^2(a,b)^2}{ab}\Big)^{-1/4+\varepsilon} \frac{q(a,b)}{\sqrt{ab}}+ q^{\varepsilon}\mathcal{R} \Big)\Big ).
\end{split}
\end{align}

  Applying \eqref{ubounds} and summing trivially, we see that upon taking $M$ large enough, 
  the error term in the last expression of \eqref{LNsquaresum} is $\ll  q^{1-\varepsilon}$. 
  Note that the main term in the last expression of \eqref{LNsquaresum} is again independent of $t$.  Proceeding as in the proof of \cite[Proposition 3.4]{GHH} yields 
\begin{align*}
\begin{split}
 \sumstar_{\substack{ \chi \shortmod q }}|L(\tfrac{1}{2}+it,f \otimes \chi)|^2 \Big (\prod^v_{j=1}|\mathcal{N}_j(t, \chi, k-1)|^2 \Big )|{\mathcal
 Q}_{v+1}(t, \chi, k)|^{2}
 \ll & \phis(q)e^{-\ell_{v+1}/2}(\log q)^{k^2}.
\end{split}
\end{align*}
 As the sum over $e^{-\ell_j/2}$ converges, we deduce \eqref{sumLsquareNQ} readily from the above, completing the proof of the proposition.

\vspace*{.5cm}

\noindent{\bf Acknowledgments.} P. G. is supported in part by NSFC grant 12471003 and L. Z. by the FRG Grant PS75536 at the University of New South Wales.

\bibliography{biblio}
\bibliographystyle{amsxport}

\end{document}